\documentclass[10pt,english]{amsart}
\usepackage[T1]{fontenc}
\usepackage[latin1]{inputenc}
\usepackage{amstext}
\usepackage{amsthm}
\usepackage{amssymb}

\makeatletter

\providecommand{\tabularnewline}{\\}

\numberwithin{equation}{section}
\numberwithin{figure}{section}
\theoremstyle{plain}
\newtheorem{thm}{\protect\theoremname}
\theoremstyle{plain}
\newtheorem{prop}[thm]{\protect\propositionname}
\theoremstyle{remark}
\newtheorem{rem}[thm]{\protect\remarkname}
\theoremstyle{plain}
\newtheorem{cor}[thm]{\protect\corollaryname}
\theoremstyle{definition}
\newtheorem{example}[thm]{\protect\examplename}
\theoremstyle{plain}
\newtheorem{lem}[thm]{\protect\lemmaname}
\theoremstyle{definition}
\newtheorem{defn}[thm]{\protect\definitionname}

\usepackage{tikz}
\usepackage{pgfplots}

\makeatother

\makeatother

\usepackage{babel}
\providecommand{\corollaryname}{Corollary}
\providecommand{\definitionname}{Definition}
\providecommand{\examplename}{Example}
\providecommand{\lemmaname}{Lemma}
\providecommand{\propositionname}{Proposition}
\providecommand{\remarkname}{Remark}
\providecommand{\theoremname}{Theorem}

\begin{document}

\title[Number of Kummer structures on generalized Kummer surfaces]{Number of Kummer structures and Moduli spaces of generalized Kummer surfaces}


\addtolength{\textwidth}{0mm}
\addtolength{\hoffset}{-0mm} 
\addtolength{\textheight}{0mm}
\addtolength{\voffset}{-0mm} 


\global\long\def\CC{\mathbb{C}}%
 
\global\long\def\BB{\mathbb{B}}%
 
\global\long\def\PP{\mathbb{P}}%
 
\global\long\def\QQ{\mathbb{Q}}%
 
\global\long\def\RR{\mathbb{R}}%
 
\global\long\def\FF{\mathbb{F}}%

\global\long\def\DD{\mathbb{D}}%
 
\global\long\def\NN{\mathbb{N}}%
\global\long\def\ZZ{\mathbb{Z}}%
 
\global\long\def\HH{\mathbb{H}}%
 
\global\long\def\Gal{{\rm Gal}}%
\global\long\def\OO{\mathcal{O}}%
\global\long\def\cO{\mathcal{O}}%
\global\long\def\bA{\mathbf{A}}%

\global\long\def\kP{\mathfrak{P}}%
 
\global\long\def\kQ{\mathfrak{q}}%
 
\global\long\def\ka{\mathfrak{a}}%
\global\long\def\kP{\mathfrak{p}}%

\global\long\def\cC{\mathfrak{\mathcal{C}}}%
 
\global\long\def\cM{\mathcal{M}}%
\global\long\def\cK{\mathcal{K}}%
\global\long\def\cN{\mathcal{N}}%
\global\long\def\cP{\mathcal{P}}%
\global\long\def\cH{\mathcal{H}}%

\global\long\def\a{\alpha}%
 
\global\long\def\b{\beta}%
 
\global\long\def\d{\delta}%
 
\global\long\def\D{\Delta}%
 
\global\long\def\L{\Lambda}%
 
\global\long\def\g{\gamma}%

\global\long\def\G{\Gamma}%
 
\global\long\def\d{\delta}%
 
\global\long\def\D{\Delta}%
 
\global\long\def\e{\varepsilon}%
 
\global\long\def\k{\kappa}%
 
\global\long\def\l{\lambda}%
 
\global\long\def\m{\mu}%
 
\global\long\def\o{\omega}%
 
\global\long\def\p{\pi}%
 
\global\long\def\P{\Pi}%
 
\global\long\def\s{\sigma}%

\global\long\def\S{\Sigma}%
 
\global\long\def\t{\theta}%
 
\global\long\def\T{\Theta}%
 
\global\long\def\f{\varphi}%

\global\long\def\deg{{\rm deg}}%
 
\global\long\def\det{{\rm det}}%

\global\long\def\Dem{Proof: }%
 
\global\long\def\ker{{\rm Ker\,}}%
 
\global\long\def\im{{\rm Im\,}}%
 
\global\long\def\rk{{\rm rk\,}}%
 
\global\long\def\car{{\rm car}}%
\global\long\def\fix{{\rm Fix( }}%

\global\long\def\card{{\rm Card\  }}%
 
\global\long\def\codim{{\rm codim\,}}%
 
\global\long\def\coker{{\rm Coker\,}}%
 
\global\long\def\mod{{\rm mod }}%

\global\long\def\pgcd{{\rm pgcd}}%
 
\global\long\def\ppcm{{\rm ppcm}}%
 
\global\long\def\la{\langle}%
 
\global\long\def\ra{\rangle}%

\global\long\def\Alb{{\rm Alb(}}%
 
\global\long\def\Jac{{\rm Jac(}}%
 
\global\long\def\Disc{{\rm Disc(}}%
 
\global\long\def\Tr{{\rm Tr}}%
\global\long\def\Nr{{\rm Nr}}%
 
\global\long\def\NS{{\rm NS(}}%
 
\global\long\def\Ns{{\rm NS}}%
 
\global\long\def\Pic{{\rm Pic(}}%
 
\global\long\def\FM{{\rm FM}}%

\global\long\def\Pr{{\rm Pr}}%
 
\global\long\def\rad{{\rm rad}}%

\global\long\def\Km{{\rm Km}}%
\global\long\def\rk{{\rm rk(}}%
\global\long\def\Hom{{\rm Hom(}}%
 
\global\long\def\End{{\rm End}}%
 
\global\long\def\aut{{\rm Aut}}%
 
\global\long\def\SSm{{\rm S}}%
\global\long\def\psl{{\rm PSL}}%

\global\long\def\cu{{\rm (-2)}}%
\global\long\def\aut{{\rm Aut}}%
 
\global\long\def\mod{{\rm \,mod\,}}%
 
\author{Xavier Roulleau}
\begin{abstract}
A generalized Kummer surface $X=\Km_{3}(A,G_{A})$ is the minimal
resolution of the quotient of a $2$-dimensional complex torus  by
an order $3$ symplectic automorphism group $G_{A}$. A Kummer structure
on $X$ is an isomorphism class of pairs $(B,G_{B})$ such that $X\simeq\Km_{3}(B,G_{B})$.
When the surface is algebraic, we obtain that the number of Kummer
structures is linked with the number of order $3$ elliptic points
on some Shimura curve naturally related to $A$. For each $n\in\NN$,
we obtain generalized Kummer surfaces $X_{n}$ for which the number
of Kummer structures is $2^{n}$. We then give a classification of
the moduli spaces of generalized Kummer surfaces. When the surface
is non algebraic, there is only one Kummer structure, but the number
of irreducible components of the moduli spaces of such surfaces is
large compared to the algebraic case. The endomorphism rings of the
complex $2$-tori we study are mainly quaternion orders, these orders
contain the ring of Eisenstein integers. One can also see this paper
as a study of quaternion orders $\OO$ over $\QQ$ that contain the
ring of Eisenstein integers. We obtain that such order is determined
up to isomorphism by its discriminant, and when the quaternion algebra
is indefinite, the order is principal. 
\end{abstract}

\maketitle

\section{Introduction}

A generalized Kummer surface $X=\Km_{3}(A,G_{A})$ (or simply $\Km_{3}(A)$)
is the minimal resolution of the quotient $A/G_{A}$ of a $2$-dimensional
complex torus  $A$ by an order $3$ symplectic automorphism group
$G_{A}$. Let us recall that the quotient surface $A/G_{A}$ has $9A_{2}$-singularities
($9$ cusps); the minimal resolution of each cusp is by a ${\bf A}_{2}$-configuration,
which means two smooth rational curves $C,C'$ in $X$ with intersection
$CC'=1$.

A generalized Kummer structure of $X$ is an isomorphism class of
pairs $(B,G_{B})$ of complex $2$-dimensional torus and order $3$
symplectic automorphism group $G_{B}\subset\aut(B)$ such that $\Km_{3}(B,G_{B})\simeq X$.
The generalized Kummer structures on $X$ are in bijective correspondence
with the orbits of $9{\bf A}_{2}$-configurations on $X$ under $\aut(X)$.
Some particular family of generalized Kummer surfaces is studied in
\cite{KRS}, and in \cite{RS3} some new $9{\bf A}_{2}$-configurations
are constructed, which permit to better understand the automorphism
group of $X$. 

In the present paper, we investigate the number $\mathrm{N}_{KS}(X)$
of generalized Kummer structures on $X$, and the moduli spaces of
these surfaces. While for non-algebraic complex tori, it is easy to
see that $\mathrm{N}_{KS}(X)=1$, the situation for abelian surfaces
is more complicated. In order to compute the numbers $\mathrm{N}_{KS}(X)$,
we use two constructions of abelian surfaces with an order $3$ symplectic
automorphism group. 

The first construction, by Barth \cite{Barth}, is of Hodge-theoretic
nature and deals with $2$-dimensional complex tori (for short complex
$2$-tori) $A$, algebraic or not. Barth studies the action of the
order $3$ automorphism group $G_{A}$ on the lattice $H^{2}(A,\ZZ)$.
The Picard number $\rho_{X}$ of the generalized Kummer surface $X=\Km_{3}(A)$
is computed to be $18$, $19$ or $20$. If $\rho_{X}=18$, the surface
is non-algebraic, if $\rho_{X}=20$ the surface is algebraic, but
is isolated in its moduli space. The main case of interest for us
will be when $X$ has Picard number $19$. Then the orthogonal complement
in the Néron-Severi group of $X$ of the $18$ curves above the $9$
cusps is generated by a class $L_{X}$. One has $L_{X}^{2}=0$ or
$2\mod6$ and for any integer $\ell$ in $\ZZ$ congruent to $0$
or $2$ mod $6$, there exists a generalized Kummer surface $X=\Km_{3}(A)$
with $L_{X}^{2}=\ell$; such a surface $X$ is algebraic if and only
if $L_{X}^{2}>0$. We obtain that 
\begin{thm}
Consider $X=\Km_{3}(A)$ and suppose that $\rho_{X}=19$. The Néron-Severi
group $\NS A)$ of the complex $2$-torus $A$ is generated by three
divisors with intersection matrices {\footnotesize{}
\[
\left(\begin{array}{ccc}
\frac{1}{3}(L_{X}^{2}-2) & 1 & 0\\
1 & -2 & 1\\
0 & 1 & -2
\end{array}\right)\text{ or }\left(\begin{array}{ccc}
\tfrac{1}{3}L_{X}^{2} & 0 & 0\\
0 & -2 & 1\\
0 & 1 & -2
\end{array}\right)
\]
}according if $L_{X}^{2}=2\mod6$ or $L_{X}^{2}=0\mod6$. If $L_{X}^{2}\neq0$,
the endomorphism ring $\End(A)$ of $A$ is an order in a quaternion
algebra over $\QQ$. 
\end{thm}

The above Theorem is proved in Section 2 (see Theorems \ref{thm:i)-NS(A)}
and \ref{Thm:The-endomorphism-ring}). The ring $\End(A)$ contains
the ring of Eisenstein integers (coming from the order $3$ automorphism
group). The quaternion algebra $\End(A)\otimes\QQ$ is indefinite
(i.e. there exists an embedding in $M_{2}(\RR)$) if and only if $L_{X}^{2}>0$.
 Let $\hat{A}$ be the dual complex torus to $A$. Using the knowledge
of $\NS A)$, we obtain in Theorem \ref{THEOREM:bidule} the following
result:
\begin{thm}
\label{THEOREM:NS-unique-genus} Let $(B,G_{B})$ be a generalized
Kummer structure on an algebraic generalized Kummer surface $X=\Km_{3}(A)$.
Then $B$ is isomorphic to $A$ or its dual $\hat{A}$.\\
The abelian surface $A$ admits a principal polarization (and therefore
$A\simeq\hat{A}$) if and only if $L_{X}^{2}=2\mod6$, or $3||L_{X}^{2}$.
\end{thm}

Here the symbol $a||n$ means that $a$ divides the integer $n$ and
the integers $a,\frac{n}{a}$ are coprime. 

The second construction of complex $2$-tori, the Shimura construction
\cite{Shimura}, is of arithmetic nature and will enable us to compute
the number $\mathrm{N}_{KS}(X)$ of Kummer structures effectively.
Let $A$ be a complex $2$-torus with an order $3$ symplectic automorphism
$J_{A}$ and with Picard number $\rho_{A}=3$. The endomorphism ring
of $A$ is an order in a quaternion algebra $H=\frac{(a,b)}{\QQ}$
over $\QQ$, where $H$ is generated by $\a,\b,$ with $\a^{2}=a,\,\b^{2}=b$
and $\a\b=-\b\a$. Since the ring $\End(A)$ contains the ring $\ZZ[J_{A}]$
of Eisenstein integers, there are restrictions on $H$: one can write
$H$ as $H=\frac{(-3,d_{H})}{\QQ}$, where $d_{H}$ is equal to the
discriminant $D_{H}$ of $H$ or to $\frac{1}{3}D_{H}$ according
if $3\nmid D_{H}$ or not. The integer $d_{H}$ is a square-free product
of primes $p$ with $p=2\mod3$. We obtain in Section \ref{subsec:Orders-containing-the-eisen}
that any quaternion order $\OO$ containing the ring of Eisenstein
integers $\ZZ[j]$ (with $j^{2}+j+1=0$) is either isomorphic to a
ring of the form
\[
\OO_{\mu}=\ZZ[j]\oplus\ZZ[j]\mu\phi,
\]
for a quaternion $\phi$ with $j\phi=\phi\bar{j}$, such that $\phi^{2}=d_{H}$,
and for some element $\mu\in\ZZ[j]$, or $\OO$ contains an order
$\OO_{\mu}$ with index $3$, for some $\mu\in\ZZ[j]$. We obtain
moreover that up to conjugation, there is a unique maximal order containing
$j$. Conversely, by the results of Shimura and Shimizu, for a given
order $\OO$, there exist complex tori $A$ having quaternionic multiplication
by $\OO$. 

To an order $\OO$ in an indefinite quaternion algebra, it is classically
associated the Shimura curve $\mathcal{H}/\G(\OO^{1})$, which is
the quotient of the upper half plane by the homographic action of
the group of norm $1$ elements of $\OO$. It is an irreducible component
of the coarse moduli space for abelian surfaces with quaternionic
multiplication by $\OO$. The elliptic points of order $3$ of $\mathcal{H}/\G(\OO^{1})$
are the ramification points of degree three of the natural quotient
map $\mathcal{H}\to X(\OO).$ The number $e_{3}(\OO)$ of elliptic
points of order $3$ is the number of orbits of order $3$ groups
in $\OO^{*}$ under the conjugation by the index $\leq2$ sub-group
$\OO^{1}\subset\OO^{*}$ of norm $1$ element. We obtain the following
results in Theorems \ref{THEOREM:NumberEllPts} and \ref{thm:une-ou-deux-compo-irred}:
\begin{thm}
Consider $\OO=\End(A)$. Suppose that $\cO$ is contained in a skew-field.
The number $\mathrm{N}_{KS}(X)$ of generalized Kummer structures
on the algebraic surface $X=\Km_{3}(A)$ is 
\[
\begin{array}{ll}
\tfrac{1}{2}e_{3}(\OO) & \text{if }L_{X}^{2}=2\mod6\text{ or }3|L_{X}^{2}\text{ but }9\not|L_{X}^{2},\\
2e_{3}(\OO) & \text{if }L_{X}^{2}=18\mod54.\\
\leq2e_{3}(\OO) & \text{if }L_{X}^{2}=18\text{ or }36\mod54.
\end{array}
\]
The ring $\OO$ is principal: the set of left $\OO$-ideals modulo
the principal ideals has a unique element.
\end{thm}

 Let us give an example: let $\mu\in\ZZ[j]$ such that $N=\mu\bar{\mu}$
is a square-free integer coprime to $D_{H}$ and to $3$, such that
the primes dividing $N$ are congruent to $1\mod3$, with $H$ indefinite.
Then $\OO_{\mu}$ is an Eichler order of level $N$, and if we take
$A$ such that $\End(A)\simeq\OO_{\mu}$, we obtain that $\mathrm{N}_{KS}(X)=2^{m+\e},$
where $\e\in\{-1,-2\}$ and $m$ is the number of primes dividing
$D_{H}N$ (we suppose here that $D_{H}\neq1$). In particular, the
number of Kummer structures can be made arbitrarily large, and in
two ways: by varying the indefinite quaternion algebra $H$ or the
level $N$. In \cite{CR} with D. Cartwright, we compute the number
$e_{3}(\OO)$ for the orders $\OO$ such that $3|L_{X}^{2}$. 

For $\ell\in\ZZ$ an integer such that $\ell=0\text{ or }2\mod6$,
let $\cM_{\ell}$ be the moduli space of generalized Kummer surfaces
$X=\Km_{3}(A)$ with the class $L_{X}\in\NS X)$ such that $L_{X}^{2}=\ell$.
The moduli space $\cM_{\ell}$ is one dimensional; when $\ell$ is
positive, it is dominated by some Shimura curve by a map of degree
$\mathrm{N}_{KS}(X)$, and when $\ell$ is negative, the space $\cM_{\ell}$
is isomorphic to $\PP^{1}$. The following result is in Theorem \ref{thm:une-ou-deux-compo-irred}
and the Section 5.4; it gives an answer to a question of Barth (\cite[Section 2.4, Problem]{Barth}):
\begin{thm}
\label{thm:Intro}  For $\ell>0$, the moduli space $\cM_{\ell}$
is irreducible.\\
The number of irreducible components of the moduli space $\cM_{\ell}$
goes to $\infty$ when $\ell$ tends to $-\infty$. 
\end{thm}

The assertion on $\cM_{\ell}$ with $\ell<0$, although not difficult
to prove, makes an interesting contrast with the case $\ell>0$; in
fact we obtain much more precise informations on these moduli spaces
(see below). For the proof of Theorem \ref{thm:Intro}, we use some
results of Miranda and Morrison \cite{MirMor1} on embeddings of lattices,
and some results of Hosono, Lian, Oguiso and Yau \cite{HLOY}, \cite{HLOY2}. 

 In \cite{BonGee}, Bonfanti and van Geemen also studied abelian
surfaces $A$ with an order $3$ symplectic automorphism group, in
order to obtain concrete examples of Shimura curves. Some of the results
on $\NS A),\,\End(A)$, which we mention here in case $L_{X}^{2}=0\mod6$
(for $X=\Km_{3}(A)$ algebraic), were previously obtained by them
(see also \cite{BonThe}, by Bonfanti). In \cite{Elkies2}, Elkies
also studies Shimura curves which parametrize the (classical) Kummer
surfaces, associated to abelian surfaces with quaternionic multiplications.
Using the moduli space of generalized Kummer surfaces, it seems possible
to obtain models of Shimura curves. 

The paper is structured as follows. In Section 2, we recall the link
between generalized Kummer structures on $X=\Km_{3}(A)$ and Fourier--Mukai
partners of $A$, when $A$ is algebraic. If $A$ is not algebraic,
the surface $X$ has a unique Kummer structure. For any complex torus
$A$ with Picard number $3$ with a symplectic order $3$ automorphism
group $G_{A}$, we study the class $L_{X}\in\NS X)$ which comes from
the invariant class $L_{A}\in\NS A)$ under the action of $G_{A}$,
and we describe the Néron-Severi group of $A$. We prove that, in
the algebraic case, the Néron-Severi group is unique in its genus,
and if $(B,G_{B})$ is a Kummer structure on $X=\Km_{3}(A)$, then
$B$ is isomorphic to $A$ or its dual $\hat{A}$. For any complex
$2$-tori $A$ with an order $3$ automorphism and Picard number $3$,
we then give generators of the endomorphism ring of $A$ and we obtain
that $\End(A)$ is an order in a quaternion algebra (when $L_{X}^{2}\neq0$);
we express the discriminant of the order $\End(A)$ as a function
of $L_{X}^{2}$. In Section 3, after some preliminaries on quaternions,
we describe -up to isomorphisms- all orders of quaternion algebras
which contain the ring of the Eisenstein integers. We recall the construction
by Shimura of abelian surfaces with quaternionic multiplication, and
its analogue for the non-algebraic case. In Section 4, we join the
results of Section 2 and Section 3, and obtain that the integer $L_{X}^{2}$
determines the quaternion algebra and the endomorphism ring of $A$.
We then compute the number $\mathrm{N}_{KS}(X)$ of Kummer structures
on $X$. In Section 5, we study the number of irreducible components
of the moduli space $\cM_{\ell}$ of generalized Kummer surfaces $X$
with fixed $\ell=L_{X}^{2}$. For each $\ell<0$, we can describe
the irreducible components of $\mathcal{M}_{\ell}$, using the reflexive
rank $4$ lattice $U\oplus A_{2}$ introduced by Barth in \cite{Barth}
for studying generalized Kummer surfaces, and studied by Vinberg in
\cite{Vinberg}. 

\textbf{Acknowledgements.} The author is grateful to Arthur Baragar,
Simon Brandhorst, Cédric Bonnafé, Donald Cartwright, Amir Dzambic,
Will Jagy, David Kohel, Olivier Ramaré and Alessandra Sarti for useful
email exchanges and discussions. The authors is also grateful to the
referee for his numerous comments and remarks that improved this paper.
Support from Centre Henri Lebesgue ANR-11-LABX-0020-01 is acknowledged.

\section{\label{sec:Topology-and-the}Hodge theory: Néron-Severi, Transcendental
lattice of $A$}

\subsection{Notations, conventions}

Given two integers $a,b,$ we write $a||b$ if $a$ divides $b$ and
$a$ is coprime to $b/a$. 

We take the rather useful convention of Conway and Sloane \cite[Chapter 15]{CS}
that $-1$ is a prime, and $\QQ_{-1}=\RR$; we will indicate when
we use that convention. 

For a lattice $\mathbb{L}$ and an integer $n$, $\mathbb{L}(n)$
denotes the group $\mathbb{L}$ with the intersection form of the
lattice multiplied by $n$. \\
The lattice with Gram matrix {\small{}$\left(\begin{array}{ll}
0 & 1\\
1 & 0
\end{array}\right)$} (respectively {\small{}$\left(\begin{array}{ll}
2 & 1\\
1 & 2
\end{array}\right)$}) is denoted by $U$ (respectively by $A_{2}$).

The Néron-Severi group of a surface $X$ is denoted by $\NS X)$,
it is the sub-lattice of classes of divisors in $H^{2}(X,\ZZ)$. The
Picard number of $X$, which is the rank of $\NS X)$, is denoted
by $\rho_{X}$. The set of numerical equivalence classes is denoted
by $\text{Num}(X)$. 

If $X$ is a K3 surface, we will often (by abuse) identify a $\cu$-curve
$C$ (i.e. a smooth rational curve contained in $X$ such that $C^{2}=-2$)
and its class $[C]$ in $\NS X)$. An ${\bf A}_{2}$-configuration
of $\cu$-curves on a K3 surface is the data of two $\cu$-curves
$C,C'$ such that $CC'=1$. 

A standard reference for basics about generalized Kummer surfaces
is the article \cite{Barth} of Barth; one also can refer to \cite{Barth2},
\cite{BonGee}, \cite{RS3}. In the following $A$ is a complex $2$-torus
and $G_{A}$ is a symplectic automorphism group of order $3$, generated
by an element $J_{A}$. Then $X=\Km_{3}(A,G_{A})$ (sometimes also
denoted by $X=\Km_{3}(A,J_{A})$) is the associated generalized Kummer
surface. It is a K3 surface which is the minimal resolution of the
quotient $A/G_{A}$. When there is no confusion, we sometimes denote
simply $X=\Km_{3}(A)$. 

The surface $X=\Km_{3}(A)$ has Picard number $\rho_{X}=16+\rho_{A}\geq18$.
When $\rho_{A}=3$, we denote by $L_{A}$ the generator of the invariant
part under $G_{A}$ of $\NS A)$, and by $L_{X}$ the generator of
the orthogonal complement of the $18$ $(-2)$-curves of the resolution
$X\to A/G_{A}$. When $A$ is algebraic, we take $L_{X},L_{A}$ to
be nef. 


\subsection{\label{subsec:The-N=0000E9ron-Severi-latofA}The Néron-Severi lattice
of $A$}

Let $A$ be a complex $2$-torus such that there exists an order
$3$ symplectic automorphism $J_{A}\in\aut(A)$. We suppose moreover
that $A$ has Picard number $3$ and we denote by $L_{A}$ the generator
of $\NS A)^{J_{A}}$, the invariant sub-lattice for $J_{A}$ (we take
it ample if $L_{A}^{2}>0$). Let $X$ be the associated generalized
Kummer surface (the desingularization of $A/J_{A}$) and let $L_{X}$
be the generator of the orthogonal complement of its natural $9{\bf A}_{2}$-configuration
(we take $L_{X}$ nef if $L_{X}^{2}>0$). Let 
\[
\pi:A\to A/J_{A}
\]
be the quotient map and $q:X\to A/J_{A}$ be the minimal resolution.
Since $L_{A}$ is $J_{A}$-invariant, there exist integers $u_{0},v_{0}$
such that 
\[
\pi^{*}q_{*}L_{X}=u_{0}L_{A},\,\,\pi_{*}L_{A}=v_{0}q_{*}L_{X}
\]
If $L_{A}^{2}>0$ these integers are positive since $L_{A}$ and $L_{X}$
are nef, otherwise, up to exchanging $L_{X}$ with $-L_{X}$, we can
suppose that they are also positive. Since moreover 
\[
\pi_{*}\pi^{*}(q_{*}L_{X})=3q_{*}L_{X},
\]
we get $u_{0}v_{0}=3$, thus $u_{0}\in\{1,3\}$. If $u_{0}=1$, then
\[
L_{X}^{2}=\frac{1}{3}L_{A}^{2}
\]
and if $u_{0}=3$, then 
\[
L_{X}^{2}=3L_{A}^{2}=0\mod6.
\]
By \cite[Section 1.2]{Barth}, the fixed sub-lattice $H_{2}(A,\ZZ)^{J_{A}}$
(containing $L_{A}$) is isomorphic to $U\oplus A_{2}$ i.e. there
exists a basis $\g_{1},\dots,\g_{4}$ with Gram matrix:{\footnotesize{}
\begin{equation}
\left(\begin{array}{lccc}
0 & 1 & 0 & 0\\
1 & 0 & 0 & 0\\
0 & 0 & 2 & 1\\
0 & 0 & 1 & 2
\end{array}\right).\label{eq:interMatrixGamma}
\end{equation}
}{\footnotesize\par}
\begin{prop}
\label{PROP:ClasseOfLA}There exist coprime integers $n_{1},\dots,n_{4}$
such that:\\
i) if $L_{X}^{2}=2\mod6$, one has
\[
L_{A}=\pi^{*}q_{*}L_{X}=3n_{1}\g_{1}+3n_{2}\g_{2}+3n_{3}\g_{3}+n_{4}(\g_{3}+\g_{4}),
\]
with $n_{4}\neq0\mod3$ (thus $u_{0}=1$), and then 
\[
L_{X}^{2}=6(n_{1}n_{2}+n_{3}^{2}+n_{3}n_{4})+2n_{4}^{2}=\frac{1}{3}L_{A}^{2},
\]
ii) if $L_{X}^{2}=0\mod6$, one has
\[
3L_{A}=\pi^{*}q_{*}L_{X}=3\left(n_{1}\g_{1}+n_{2}\g_{2}+n_{3}\g_{3}+n_{4}(\g_{3}+\g_{4})\right),
\]
(thus $u_{0}=3$) with $\gcd(n_{1},n_{2},n_{3},3)=1$, and then 
\[
L_{X}^{2}=6(n_{1}n_{2}+n_{3}^{2}+3n_{3}n_{4}+3n_{4}^{2})=3L_{A}^{2}.
\]
\end{prop}

\begin{rem}
Conversely the choice of such integers $(n_{1},\dots,n_{4})$ defines
a class $L$, which is the invariant polarization $L_{A}$ of some
complex torus $A$ with an order $3$ symplectic automorphism, and
therefore corresponds to some generalized Kummer surface, see \cite[Section 2.4]{Barth}.
\end{rem}

\begin{proof}
By \cite[Section 1.3, Corollary]{Barth}, and the end of \cite[Section 1.2]{Barth},
the primitive class $L_{X}$ is given by 
\[
L_{X}=n_{1}\xi_{1}+\dots+n_{4}\xi_{4},
\]
where the integers $n_{1},\dots,n_{4}$ are coprime, and $\xi_{1},\dots,\xi_{4}$
is a basis of the orthogonal in $H^{2}(X,\ZZ)$ of the $18$ $\cu$-curves
above the $9$ cusp singularities of $A/J_{A}$. The basis is such
that $\pi^{*}q_{*}\xi_{i}=3\gamma_{i}$ for $i\leq3$ and $\pi^{*}q_{*}\xi_{4}=\g_{3}+\gamma_{4}$,
thus:
\[
\pi^{*}q_{*}L_{X}=3n_{1}\g_{1}+3n_{2}\g_{2}+3n_{3}\g_{3}+n_{4}(\g_{3}+\g_{4}).
\]
Therefore if $n_{4}\neq0\mod3$, the integers $n_{1},\dots,n_{4}$
being coprime, the class $\pi^{*}q_{*}L_{X}$ is primitive, and we
have $L_{A}=\pi^{*}q_{*}L_{X}$. If $n_{4}=3n_{4}'$ for some $n_{4}'$,
we must have $\gcd(n_{1},n_{2},n_{3},3)=1$ since we suppose that
the class $L_{X}$ is primitive. Then we have 
\[
\pi^{*}q_{*}L_{X}=3\left(n_{1}\g_{1}+n_{2}\g_{2}+n_{3}\g_{3}+n_{4}'(\g_{3}+\g_{4})\right),
\]
thus $L_{A}=\frac{1}{3}\pi^{*}q_{*}L_{X}$ and the result follows
by formally replacing $n_{4}'$ by $n_{4}$ in the notations, since
we will not be using again the class $L_{X}$ in basis $\xi_{1},\dots,\xi_{4}$.
\end{proof}
According to \cite[Section 2.4, first Proposition]{Barth}, the Néron-Severi
lattice of $A$ contains the lattice $M_{0}$ generated by the invariant
class $L_{A}$ and classes $\d_{1},\d_{2}$, such that $L_{A},\d_{1},\d_{2}$
have Gram matrix{\footnotesize{}
\[
Gr=\left(\begin{array}{ccc}
L_{A}^{2} & 0 & 0\\
0 & -2 & 1\\
0 & 1 & -2
\end{array}\right).
\]
}The classes $\d_{1},\d_{2}$ are non-invariant for the action of
$J_{A}$, in fact: $\d_{2}=J_{A}^{*}\d_{1}$. Let us prove:
\begin{thm}
\label{thm:i)-NS(A)}According to the cases $L_{X}^{2}=2$ or $0\mod6$,
the Néron-Severi lattice of $A$ and the order of its discriminant
group are as follows:

\begin{tabular}{|c|c|c|c|c|}
\hline 
 & $\Ns(A)$ & $L_{X}^{2}$ & $L_{A}^{2}$ & $|\Disc\Ns(A))|$\tabularnewline
\hline 
i) & $\left(\begin{array}{ccc}
2k & 1 & 0\\
1 & -2 & 1\\
0 & 1 & -2
\end{array}\right)$ & $L_{X}^{2}=6k+2$ & $L_{A}^{2}=18k+6=3L_{X}^{2}$ & $L_{X}^{2}$\tabularnewline
\hline 
ii) & $\left(\begin{array}{ccc}
2k' & 0 & 0\\
0 & -2 & 1\\
0 & 1 & -2
\end{array}\right)$ & $L_{X}^{2}=6k'$ & $L_{A}^{2}=2k'=\frac{1}{3}L_{X}^{2}$ & $L_{X}^{2}$\tabularnewline
\hline 
\end{tabular}

In case i), the lattice $\NS A)$ is generated by $D_{A}=\frac{1}{3}(L_{A}-(2\d_{1}+\d_{2}))$
and $\d_{1},\d_{2}$, moreover the discriminant group of $\NS A)$
is cyclic of order $L_{X}^{2}=\frac{1}{3}L_{A}^{2}$. In case ii),
the Néron-Severi lattice of $A$ is generated by $L_{A}$ and $\d_{1},\d_{2}$.
\end{thm}

\begin{rem}
\label{rem:As-observed-byBarth}As observed by Barth, for any integer
$\ell$ with $\ell=6k+2$ or $\ell=6k$ (for some $k\in\ZZ$), there
exists a generalized Kummer surface $X$ with $L_{X}^{2}=\ell$. Indeed,
according if $L_{X}^{2}=2\mod6$ or $L_{X}^{2}=0\mod6$, one has 
\[
L_{X}^{2}=6(n_{1}n_{2}+n_{3}^{2}+n_{3}n_{4})+2n_{4}^{2}\text{ or }L_{X}^{2}=6(n_{1}n_{2}+n_{3}^{2}+3n_{3}n_{4}+3n_{4}^{2}).
\]
 By taking $(n_{1},\dots,n_{4})=(k,1,0,1)$ or $(k,1,0,0)$, we get
all possible $L_{X}^{2}$.
\end{rem}

\begin{proof}
(of Theorem \ref{thm:i)-NS(A)}). The discriminant group of the lattice
$M_{0}$ is generated by $\frac{1}{L_{A}^{2}}L_{A},\,\frac{1}{3}(\d_{1}+2\d_{2})$;
it is isomorphic to $\ZZ/L_{A}^{2}\ZZ\times\ZZ/3\ZZ$. Since $L_{A}$
is primitive, the only possibility for having $\NS X)\neq M_{0}$
is that $3|L_{A}^{2}$ and to construct a class in $\NS A)$ with
the elements $\frac{1}{3}L_{A}$ and $\frac{1}{3}(\d_{1}+2\d_{2})$
of the discriminant group. In \cite[Sections 1.2 \& 2.2]{Barth},
Barth expresses the classes $\d_{1},\d_{2}$ and $L_{A}$ in the basis
\[
\a_{1}\wedge\a_{2},\a_{1}\wedge\b_{1},\a_{1}\wedge\b_{2},\a_{2}\wedge\b_{1},\a_{2}\wedge\b_{2},\b_{1}\wedge\b_{2}
\]
of $H^{2}(A,\ZZ)$, where $\a_{1},\a_{2},\b_{1},\b_{2}$ are generators
of $H_{1}(A,\ZZ)$, and we have 
\[
\begin{array}{l}
\d_{1}=\a_{1}\wedge\a_{2}-\b_{1}\wedge\b_{2},\\
\d_{2}=\a_{1}\wedge\b_{2}-\a_{2}\wedge\b_{1}+\b_{1}\wedge\b_{2}
\end{array}.
\]
Suppose that $L_{X}^{2}=2\mod6$. Then $3|L_{A}^{2}$ and by Lemma
\ref{PROP:ClasseOfLA}:
\[
L_{A}=3n_{1}\g_{1}+3n_{2}\g_{2}+3n_{3}\g_{3}+n_{4}(\g_{3}+\g_{4})
\]
with $n_{4}\neq0\mod3$. By \cite[Section 1.2]{Barth}, we have 
\[
\begin{array}{l}
\g_{1}=-\a_{1}\wedge\b_{1}\\
\g_{2}=\a_{2}\wedge\b_{2}\\
\g_{3}=\a_{1}\wedge\b_{2}+\a_{2}\wedge\b_{1}\\
\g_{4}=\a_{1}\wedge\a_{2}+\a_{1}\wedge\b_{2}+\b_{1}\wedge\b_{2}
\end{array}
\]
and therefore in $H^{2}(A,\ZZ/3\ZZ)$, we have 
\[
\begin{array}{l}
L_{A}-(\d_{1}+2\d_{2})=(n_{4}-1)\a_{1}\wedge\a_{2}+(2n_{4}-2)\a_{1}\wedge\b_{2}+(n_{4}+2)\a_{2}\wedge\b_{1}+(n_{4}-1)\b_{1}\wedge\b_{2},\\
L_{A}-(2\d_{1}+\d_{2})=(n_{4}-2)\a_{1}\wedge\a_{2}+(2n_{4}-1)\a_{1}\wedge\b_{2}+(n_{4}+1)\a_{2}\wedge\b_{1}+(n_{4}+1)\b_{1}\wedge\b_{2}.
\end{array}
\]
We conclude that the class $\frac{1}{3}(L_{A}-(\d_{1}+2\d_{2}))$
(respectively $\frac{1}{3}(L_{A}-(2\d_{1}+\d_{2})$) is in the Néron-Severi
lattice if and only if $n_{4}=1\mod3$ (respectively $n_{4}=2\mod3$).
We take the freedom to permute the role of $\d_{1},\d_{2}$ if necessary,
so that we have the intersection matrix given in the statement of
Theorem \ref{thm:i)-NS(A)}, for both cases $n_{4}=1$ and $n_{4}=2\mod3$.
That proves the result when $L_{X}^{2}=2\mod6$.\\
Suppose now that $L_{X}^{2}=0\mod6$. If $3\nmid L_{A}^{2}$, since
$L_{A}$ is primitive, the Néron-Severi group must be generated by
$L_{A},\d_{1},\d_{2}$. Suppose $3|L_{A}^{2}$ (then $9|L_{X}^{2}$);
since by Proposition \ref{PROP:ClasseOfLA} we have
\[
L_{A}=n_{1}\g_{1}+n_{2}\g_{2}+n_{3}\g_{3}+n_{4}(\g_{3}+\g_{4})
\]
and we know the classes $\g_{i}$, one can check that the only possibility
for having a larger lattice is that $n_{1}=n_{2}=n_{3}=0\mod3$ and
$n_{4}\neq0\mod3$. But by Proposition \ref{PROP:ClasseOfLA}, case
ii), the integers $n_{1},n_{2},n_{3}$ must be coprime to $3$. 
\end{proof}

\subsection{\label{subsec:GenusofNSA} The genus of the Néron-Severi lattice
of $A$}

In this section, we suppose that the surfaces $X,A$ are algebraic. 

For an even lattice $S$, we denote by $\mathrm{A}_{S}$ the discriminant
group $S^{*}/S$ and by $q_{\mathrm{A}{}_{S}}:\mathrm{A}_{S}\to\QQ/2\ZZ$
the quadratic form on $\mathrm{A}_{S}$ induced by the bilinear form
on $S$. The genus $\mathcal{G}(S)$ of $S$ is the set of isomorphism
classes of lattices $S'$ of same rank, same signature, and with isometric
discriminant group: $(\mathrm{A}_{S'},q_{\mathrm{A}{}_{S'}})\simeq(\mathrm{A}{}_{S},q_{\mathrm{A}_{S}})$;
it is a finite set. Let us denote by $O(\mathrm{A}{}_{S})$ the automorphism
group of $\mathrm{A}{}_{S}$ preserving the quadratic form $q_{\mathrm{A}{}_{S}}$.
The orthogonal group $O(S)$ acts on $\mathrm{A}{}_{S}$ through the
natural morphism $O(S)\to O(\mathrm{A}{}_{S})$. One has:
\begin{thm}
\label{thm:UniqueinGenus}Let $X=\Km_{3}(A)$ be an algebraic Kummer
surface. The genus of the lattice $\NS A)$ is $\{\NS A)\}$ and the
natural map $O(\NS A))\to O(\mathrm{A}{}_{\NS A)})$ is surjective. 
\end{thm}

\begin{proof}
Let us recall that, by definition, the length of a finite abelian
group $\mathrm{A}$ is the minimal number of a generating set of $\mathrm{A}$.
If $L_{X}^{2}=2\mod6$, the discriminant group of the rank $3$ lattice
$\NS A)$ is cyclic, thus it has length $\ell=1$. If $3|L_{X}^{2}$
but $9\not|L_{X}^{2}$, then $L_{A}^{2}=\frac{1}{3}L_{X}^{2}$ is
coprime to $3$ and therefore the discriminant group of $\NS A)$
(of order $3L_{A}^{2}=L_{X}^{2})$ has also length $\ell=1$. Then
in these two cases, the inequality 
\[
\text{rank}(\NS A))=3\geq2+\ell
\]
is satisfied and we can apply \cite[Theorem 1.14.2]{Nikulin} to conclude
that the genus of $\NS A)$ is $\{\NS A)\}$ and that the natural
map $O(S)\to O(\mathrm{A}{}_{S})$ is surjective for $S=\NS A)$. 

Suppose now that $9|L_{X}^{2}$; the quadratic form associated to
$\NS A)$ is (in the basis $L_{A},\d_{1},\d_{2}$):
\[
Q(x,y,z)=2(3kx^{2}-y^{2}+yz-z^{2}),
\]
(for $k=\tfrac{1}{9}L_{X}^{2}>0$) and the discriminant group $\mathrm{A}_{\NS A)}$
is (isomorphic to) $\ZZ/3\ZZ\times\ZZ/3k\ZZ$. By \cite[Chapter VIII, Lemma 7.7(1)]{MirMor},
the quadratic form $Q$ is $2$-regular, by \cite[Chapter VIII, Lemma 7.6(2)]{MirMor}
$Q$ is $3$-semiregular, and by \cite[Chapter VIII, Lemma 7.6(1)]{MirMor}
it is $p$-regular for any prime $p\geq5$. One can therefore apply
\cite[Chapter VIII, Theorem 7.5]{MirMor} to conclude that also in
that case the genus of $\NS A)$ is $\{\NS A)\}$ and that the natural
map $O(\NS A))\to O(\mathrm{A}{}_{\NS A)})$ is surjective. 
\end{proof}

\subsection{\label{subsec:The-transcendent-lattice}The transcendental lattice
of $A$}

Let us recall that Proposition \ref{PROP:ClasseOfLA} gives the $J_{A}$-invariant
class of $L_{A}$ in the basis $\g_{1},\dots,\g_{4}$ of the sub-lattice
$H^{2}(A,\ZZ)^{J_{A}}\simeq U\oplus A_{2}$ which is invariant under
the order $3$ automorphism $J_{A}$. The transcendental lattice $T(A)=\NS X)^{\perp}$
of $A$ is also the orthogonal complement of $L_{A}$ in $H^{2}(A,\ZZ)^{J_{A}}$.
We have:
\begin{cor}
\label{cor:Bidule2}Suppose that $A$ is algebraic. The lattice $T(A)(-1)$
is isomorphic to $\NS A)$, in particular $T(A)$ is unique in its
genus.
\end{cor}

\begin{proof}
The lattices $\NS A)$ and $T(A)$ are orthogonal sub-lattices in
$H^{2}(A,\ZZ),$ which is isomorphic to $U^{\oplus3}$ and unimodular,
thus their discriminant groups $\mathrm{A}_{\text{NS}(A)}$ and $\mathrm{A}{}_{T(A)}$
must be isomorphic, and the quadratic forms on $\mathrm{A}_{\text{NS}(A)}$
and $\mathrm{A}{}_{T(A)}$ have the opposite sign. Moreover these
two lattices are of the same rank, with respective signatures $(1,2)$,
$(2,1)$. These conditions imply that $T(A)(-1)$ must be in the genus
of $\NS A)$. Since by Theorem \ref{THEOREM:bidule} the genus contains
only $\NS A)$, the lattice $T(A)(-1)$ is isomorphic to $\NS A)$.
\end{proof}

\subsection{\label{subsec:Fourier-Mukai-partners-and}Fourier--Mukai partners
and generalized Kummer structures}

\subsubsection{Integral Hodge structures, Fourier--Mukai partner, dual abelian
surface}

Recall that for a K3 surface (resp. a complex torus) $X$, a \textit{Fourier--Mukai
partner} of $X$ is a K3 surface (resp. a complex torus) $Y$ such
that there is an isomorphism of Hodge structures
\[
(T(Y),\CC\o_{Y})\simeq(T(X),\CC\o_{X}),
\]
where $\o_{X}$ is a generator of $H^{0}(X,K_{X})$, and $T(X)=\NS X)^{\perp}$
is the transcendental lattice. The set of isomorphism classes of Fourier--Mukai
partners of $X$ is denoted by $\text{FM}(X)$. 

Let $A$ be a complex torus, let $G_{A}$ be an order $3$ automorphism
group of $A$ acting symplectically and let $X=\Km_{3}(A,G_{A})$
be the associated generalized Kummer surface. An isomorphism class
of pairs $(B,G_{B})$ where $B$ is a complex torus (and $G_{B}$
an order $3$ symplectic automorphism group) such that $\Km_{3}(B,G_{B})\simeq X$
is called a \textit{generalized Kummer structure} on $X$. We denote
by $\mathcal{K}(X)$ the set of such isomorphism classes.

Consider the following two assertions:\\
(I) The surface $B$ is a Fourier--Mukai partner of $A$ (i.e. $B\in\text{FM}(A)$).
\\
(II) The surfaces $\Km_{3}(B)$ and $\Km_{3}(A)$ are isomorphic (i.e.
$(B,G_{B})\in\mathcal{K}(X)$). 

The following is the main result of \cite{RS4}:
\begin{thm}
\label{thm:Fourier-Mukai}Let $X=\Km_{3}(A)$ be an algebraic generalized
Kummer surface with Picard number $19$.\\
Suppose that $L_{X}^{2}\neq0,36\mod54$. Then (I) is equivalent to
(II).\\
Suppose that $L_{X}^{2}=0\text{ or }36\mod54$. Then (II) implies
(I), but (I) does not imply (II) in general.
\end{thm}

\begin{rem}
By \cite[Proposition 5.3]{BriMac}, the number $|\text{FM}(A)|$ of
Fourier--Mukai partners of an abelian surface $A$ is finite. 
\end{rem}

 Let $A$ be an abelian surface with Picard number $3$ that possesses
an order $3$ symplectic automorphism $J_{A}$. We denote by $\hat{J}_{A}$
the symplectic order $3$ automorphism of the dual complex abelian
surface $\hat{A}=$$\text{Pic}^{0}(A)$, induced by the action $\mathcal{L}\to J_{A}^{*}\mathcal{L}.$
Let $X=\Km_{3}(A)$ be the associated generalized Kummer surface.
\begin{cor}
\label{cor:DualAbel}Suppose that $L_{X}^{2}\neq0,36\mod54$. Then
$\Km_{3}(A)$ and $\Km_{3}(\hat{A})$ are isomorphic: $(\hat{A},\hat{J}_{A})\in\mathcal{K}(X)$.
\end{cor}

\begin{proof}
By Shioda's Torelli Theorem \cite{Shioda} recalled in Section 5,
there exists an isomorphism of Hodge structures
\[
(T(\hat{A}),\CC\o_{\hat{A}})\simeq(T(A),\CC\o_{A}).
\]
Therefore $A$ and $\hat{A}$ are Fourier--Mukai partners and, using
the hypothesis on $L_{X}^{2}$, we can apply Theorem \ref{thm:Fourier-Mukai}
to the pairs $(A,J_{A}),$ $(\hat{A},\hat{J}_{A})$ to conclude that
$(\hat{A},\hat{J}_{A})\in\mathcal{K}(X)$.
\end{proof}
\begin{rem}
As we will see below, the abelian surface $A$ has a principal polarization
if and only if $L_{X}^{2}=2\mod6$ or $L_{X}^{2}=6k$ with $k\neq0\mod3$.
In that case, let $\phi$ be an isomorphism between $A$ and $\hat{A}$.
The pairs $(\hat{A},\hat{J}_{A})$ and $(A,J_{A}')$ with $J'_{A}=\phi^{-1}\hat{J_{A}}\phi$
define the same Kummer structure, but we do not know if $(A,J_{A}')$
define the same Kummer structure as $(A,J_{A})$. 
\end{rem}

\subsubsection{The generalized Kummer structures come from Fourier--Mukai partners}

Let $A$ be an abelian surface; for a lattice $S$, let us denote
by $\FM(A,S)$ the set of Fourier--Mukai partners $B$ of $A$ such
that $\NS B)\simeq S$. The following result is a consequence of Theorems
\ref{thm:Fourier-Mukai} and \ref{thm:i)-NS(A)}:
\begin{cor}
\label{cor:FMpartnerNS}Let $X=\Km_{3}(A)$ be an algebraic generalized
Kummer surface with Picard number $19$. Let $(B,G_{B})$ be a generalized
Kummer structure of $X$. Then $B\in\FM(A,\NS A))$.
\end{cor}

\begin{proof}
By Theorem \ref{thm:Fourier-Mukai}, the surface $B$ is a Fourier--Mukai
partner of $A$. The rank $18$ orthogonal complement of $L_{X}^{2}$
in $\NS X)$ is -up to isomorphism- independent of $L_{X}^{2}$ (see
\cite{RS3}). That implies that $L_{X}^{2}=L_{Y}^{2}$ for $Y=\Km_{3}(B)\simeq\Km_{3}(A)$.
By Theorem \ref{thm:i)-NS(A)}, the Néron-Severi group of $A$ (respectively
$B$) is uniquely determined by the integer $L_{X}^{2}$ (respectively
$L_{Y}^{2}$), and therefore $\NS A)\simeq\NS B)$. 
\end{proof}
For a $2$-dimensional complex torus  $A$, let us denote by $G_{\text{Ho}}$
the group of Hodge isometries of the transcendental  lattice $T(A)$.
It acts on $O(\mathrm{A}{}_{S})$ through the natural isomorphism
$(\mathrm{A}{}_{S},-q_{\mathrm{A}{}_{S}})\simeq(\mathrm{A}{}_{T},q_{\mathrm{A}{}_{T}})$,
where $S=\NS A)$ and $T=T(A)=\NS A)^{\perp}$ are contained in $H^{2}(A,\ZZ)\simeq U^{\oplus3}$.
Related to Corollary \ref{cor:FMpartnerNS}, let us recall the following
result:
\begin{thm}
\label{thm:()-HOLOY2}(\cite[Theorem 2.3]{HLOY2} and \cite[Section 2, Claim]{HLOY})
Let $\mathcal{P}_{A}$ be the double coset $\mathcal{P}_{A}=O(\NS A))\setminus O(\mathrm{A}{}_{\NS A)})/G_{\text{Ho}}$.
There exists a map 
\[
\xi:\FM(A,\NS A))\to\mathcal{P}_{A},
\]
which is onto and such that $\xi^{-1}(\xi([B]))=\{B,\hat{B}\}$.
\end{thm}

Let $\hat{A}$ be the dual torus of $A$. Define $X=\Km_{3}(A)$. 
\begin{thm}
\label{THEOREM:bidule}We have $\FM(A)=\FM(A,\NS A))=\{A,\hat{A}\}$,
in particular if $(B,G_{B})$ is a generalized Kummer structure, then
either $B\simeq A$ or $B\simeq\hat{A}$. \\
The surface $A$ has a principal polarization if and only if $L_{X}^{2}=2\mod6$
or $3||L_{X}^{2}$. \\
Suppose that $L_{X}^{2}=6k$ with $k=3\mod9$. Then $(A,J_{A})$ and
$(\hat{A},\hat{J}_{A})$ are two distinct Kummer structures on $\Km_{3}(A)$. 
\end{thm}

\begin{proof}
By Theorem \ref{thm:UniqueinGenus}, the map $O(\NS A))\to O(\mathrm{A}{}_{\NS A)})$
is onto, therefore $|\mathcal{P}_{A}|=1$ and by Theorem \ref{thm:()-HOLOY2},
$\text{FM}(A,\NS A))$$=\{A,\hat{A}\}$. If $B$ is a Fourier--Mukai
partner of $A$, then $T(B)\simeq T(A)$. By Corollary \ref{cor:Bidule2},
one has $T(A)(-1)\simeq\NS A)$ and $T(B)(-1)\simeq\NS B)$, thus
$B\in\text{FM}(A,\NS A))$. We thus obtain that 
\[
\text{FM}(A)=\text{FM}(A,\NS A))=\{A,\hat{A}\}.
\]
By Corollary \ref{cor:FMpartnerNS}, if $(B,G_{B})$ is a generalized
Kummer structure one has $B\in\FM(A)$, thus either $B\simeq A$ or
$B\simeq\hat{A}$.

Suppose that $L_{X}^{2}=2\mod6$ or $3||L_{X}^{2}$. If the quadratic
forms
\begin{equation}
(xD_{A}+y\d_{1}+z\d_{2})^{2}\text{ or }(xL_{A}+y\d_{1}+z\d_{2})^{2}\label{eq:lxeq2mod6}
\end{equation}
(according to the cases) associated to the lattice $\NS A)$ represents
$2$, then the surface $A$ has a principal polarization, and therefore
$A$ is isomorphic to its dual $\hat{A}$. We know that this quadratic
form is unique in its genus, therefore one can apply the Hasse-Minkowsky
principle: if there exists a local solution for each prime and over
$\RR,$ then there exists an integral (global) solution. That can
be checked using \cite[Chapitre 4]{Serre}. We leave it to the reader;
a complete proof may be found on arXiv version 1 of this paper.

Suppose now that $9|L_{X}^{2}$ and let $k'\in\NN$ be such that $L_{A}^{2}=3k'$.
The integers represented by the quadratic form $y^{2}-yz+z^{2}$ are
congruent to $0$ or $1\mod3$, thus the equation $3k'x^{2}-(y^{2}-yz+z^{2})=1$
has no solution modulo $3$, and therefore, there is no solution to
equation $(xL_{A}^{2}+y\d_{1}+z\d_{2})^{2}=2$: $A$ has no principal
polarization. 

Suppose that $L_{X}^{2}=6k$ with $k=3\mod9$. By Shioda's Torelli
Theorem \cite{Shioda}, $A$ and $\hat{A}$ are Fourier--Mukai partners.
Then by implication (I) $\Rightarrow$ (II) of Theorem \ref{thm:Fourier-Mukai},
the pair $(\hat{A},\hat{J}_{A})$ is a generalized Kummer structure
of $X=\Km(A,J_{A})$. Since $A$ and $\hat{A}$ are not isomorphic,
the generalized Kummer structures $(A,J_{A})$ and $(\hat{A},\hat{J}_{A})$
are distinct. 
\end{proof}
\begin{rem}
In case $k=0\mod9$ or $k=6\mod9$, we do not know if $(\hat{A},\hat{J}_{A})$
is a generalized Kummer structure of $X=\Km(A,J_{A})$.
\end{rem}

\subsection{Non-algebraic case: a unique generalized Kummer structure}

Let $(A,G_{A}),(B,G_{B})$ be two non-algebraic complex $2$-tori
with an order $3$ automorphism group. 
\begin{prop}
Suppose that $\Km_{3}(A)\simeq\Km_{3}(B)$. Then there exists an isomorphism
$\tau:A\to B$ such that $\tau^{-1}G_{B}\tau=G_{A}$. In particular
$\Km_{3}(A)$ has a unique generalized Kummer structure. 
\end{prop}

\begin{proof}
The Néron-Severi group of $X=\Km_{3}(A)$ has rank $\rho_{X}=18$
or $19$. The lattice $\NS X)$ is described in \cite[Section 2]{RS3},
(using \cite[Proof of Proposition 1.3]{Bertin}), where this is stated
for abelian surfaces, but it is valid more generally for complex $2$-tori
since the proof only uses algebra. With the notations of \cite{RS3},
in case $\rho_{X}=18$, one has $\NS X)=\mathcal{K}_{3}$, otherwise
$\NS X)$ is given in \cite[Theorem 2]{RS3}. 

Suppose that $\NS X)$ is negative definite (this is equivalent to
suppose that the map $\NS A)\to\text{Num}(A)$ has trivial kernel).
By computing the roots (the elements of square $-2$) of $\NS X)$,
one finds that in both cases $\rho_{X}=18$ or $19$, the only irreducible
curves on $X$ are the $18$ $(-2)$-curves coming from the desingularization
of $A/G_{A}$. Since $H_{1}(X,\ZZ)=\{0\}$ on K3 surfaces, the cyclic
triple cover $\eta_{A}$ of $X$ branched over the $2\cdot9$ curves
in the exceptional locus of the minimal resolution $X\to A/G_{A}$
is unique (see \cite[Chapter I, Lemma 17.1]{BHPV}), and from that
triple cover we may recover $A$. Since there are only $18$ curves
on $\Km_{3}(A)$ and $\Km_{3}(B)$, the isomorphism $\phi:\Km_{3}(B)\to X=\Km_{3}(A)$
must send the $18$ $\cu$-curves of $\Km_{3}(B)$ to the $18$ $\cu$-curves
of $X$. The maps $\phi\circ\eta_{B}$ and $\eta_{A}$ are triple
covers with same branch locus, therefore by the unicity of the triple
cover, the surface $B$ is isomorphic to $A$ and we obtain that $\Km_{3}(A)$
possesses a unique generalized Kummer structure. 

Suppose that the map $\NS A)\to\text{Num}(A)$ has a non-trivial kernel.
Let $F$ be a generator of that kernel. One has $F^{2}=0$, therefore
there exists a fibration $\varphi:A\to E$ of the complex torus onto
an elliptic curve. The fibration $\varphi$ induces an elliptic fibration
$\varphi':X=\Km_{3}(A)\to\PP^{1}$ and the $9$ fixed points of $G_{A}$
(which are $3$-torsion points of $A$) induce singular fibers on
$X$. Using Kodaira's Table \cite[Chapter V, Section 7]{BHPV}, the
singular fibers $F$ of an elliptic fibration have Euler number $e(F)$
as follows:\\
\begin{tabular}{|c|c|c|c|c|c|c|c|c|}
\hline 
$F$ & $I_{n},n\geq1$ & $II$ & $III$ & $IV$ & $I_{n}^{*},n\geq0$ & $\tilde{{\bf E}}_{6}$ & $\tilde{{\bf E}}_{7}$ & $\tilde{{\bf E}}_{8}$\tabularnewline
\hline 
$e(F)$ & $n$ & $2$ & $3$ & $3$ & $6+n$ & $8$ & $9$ & $10$\tabularnewline
\hline 
\end{tabular}\\
The fibers $I_{1},I_{2},II,III$ do not contain a ${\bf A}_{2}$-configuration
i.e. two $\cu$-curves $C,C'$ such that $CC'=1$ (see Kodaira's Table)
and the Euler number of a fiber containing (at least) a ${\bf A}_{2}$
configuration is $\geq3$. By \cite[Chapter III, Proposition 11.4]{BHPV},
since the Euler number of a K3 surface $X$ is $24$, one has 
\[
24=\sum_{s\in\PP^{1}}e(F_{s})
\]
where $F_{s}$ is the fiber over $s$ of an elliptic fibration $X\to\PP^{1}$
(one has $e(F_{s})=0$ if $F_{s}$ is smooth). Therefore, for example,
the nine ${\bf A}_{2}$-configurations cannot be spread into $9$
distinct fibers of $\varphi'$. Also from the Kodaira Table, the $9$
fixed points of $G$ cannot belong to the same fiber of $\varphi$.
Using that, and the fact that the image of a $3$-torsion group by
$\varphi$ must be a $3$-torsion group, we obtain that the unique
possibility for the $9$ fixed points of $G_{A}$ is that they are
distributed by $3$ on $3$ fibers on $A$, and that the elliptic
fibration $\varphi'$ has $3$ singular fibers of type $\tilde{\mathbf{E}}_{6}$
(each such fiber has $7$ irreducible components), and $\varphi'$
has no other singular fibers. 

There cannot be curves that are multi-sections, otherwise there would
exist a divisor of square $>0$. Thus the surface $X$ contains exactly
$21$ $\cu$-curves. Moreover, using Kodaira's Table, among these
$21$ curves, there is a unique set of $18$ curves that form $9$
disjoint $\mathbf{A}_{2}$-configurations. Then we can use the same
argument as in case when $\NS X)$ is negative definite to conclude
that $\Km_{3}(A)$ has also a unique generalized Kummer structure. 
\end{proof}

\subsection{\label{subsec:The-endomorphism-ring}The endomorphism ring of $A$}

Let $(A,J_{A})$ be a complex torus and an order $3$ automorphism.
Let $X=\Km_{3}(A)$ be the associated generalized Kummer surface.
For $a,b\in\QQ$, $ab\neq0$, let us denote by $\frac{(a,b)}{\QQ}$
the quaternion algebra generated by $r,\phi$ such that $r^{2}=a,\phi^{2}=b,$
$r\phi=-\phi r$. Let us prove the following result 
\begin{thm}
\label{Thm:The-endomorphism-ring}i) Suppose that $L_{X}^{2}=2\mod6$.
The endomorphism ring of $A$ contains elements $j=J_{A},r,\phi,\psi$,
such that 
\[
r=1+2j,\,\,\,r^{2}=-3,\,\,\,\phi^{2}=\frac{1}{2}L_{X}^{2},\,\,\,r\phi=-\phi r,\,\,\,\psi=\frac{r}{3}(\phi-1),
\]
with $\psi^{2}=\frac{1}{6}(L_{X}^{2}-2)\in\End(A)$. Moreover for
general $A$, we have
\[
\End(A)=\ZZ[j,\psi],\,\,\,\,\End_{\QQ}(A)=\frac{(-3,\frac{1}{2}L_{X}^{2})}{\QQ}.
\]
ii) Suppose that $L_{X}^{2}=0\mod6$. The endomorphism ring of $A$
contains elements $j=J_{A},r,\phi$, such that 
\[
r=1+2j,\,\,\,r^{2}=-3,\,\,\,\phi^{2}=\frac{1}{6}L_{X}^{2},\,\,\,r\phi=-\phi r.
\]
Moreover for general $A$, we have
\[
\End(A)=\ZZ[j,\phi],\,\,\text{and if }L_{X}^{2}\neq0\text{ then }\,\End_{\QQ}(A)=\frac{(-3,\frac{1}{6}L_{X}^{2})}{\QQ}.
\]
In both cases i) and ii) with $L_{X}^{2}\neq0$, the discriminant
of the quaternion order $\End(A)$ is equal to $\frac{1}{2}L_{X}^{2}$. 
\end{thm}

\begin{example}
Suppose that $L_{A}^{2}=L_{X}^{2}=0$ (for example take $L_{A}=\g_{2}$).
Since $L_{A}^{2}=0,L_{A}\d_{1}=L_{A}\d_{2}=0$, the map $\NS X)\to\text{Num}(X)$
(where $\text{Num}(X)$ is the group of numerical equivalence classes)
has a one dimensional kernel generated by $L_{A}$. Since $\text{Num}(X)$
is negative definite, the torus $A$ is not algebraic. The class $L_{A}$
is the class of an elliptic curve $E$, and there is an extension
\[
0\to E\to A\to E'\to0,
\]
where $E'$ is also an elliptic curve; this is a so-called Shafarevich
extension, see \cite[Proposition 7.1]{BirLanCT}. The endomorphism
$\phi$ is non-zero and such that $\phi^{2}=0$ and $r\phi=-\phi r$.
\end{example}

\begin{rem}
In his thesis \cite{BonThe} Bonfanti, and in \cite{BonGee}, Bonfanti
and van Geemen obtain the construction of some $(1,d)$-polarized
abelian surfaces with an order $3$ symplectic action $J_{A}$. These
surfaces $(A,J_{A})$ have a deformation to $E_{j}\times E_{j}$,
where $E_{j}=\CC/\ZZ[j]$ for $j^{2}+j+1=0$, and $J_{E_{j}\times E_{j}}=\text{Diag}(j,j^{2})$.
They also prove that the abelian surfaces they study have endomorphism
ring $\End(A)$ generated by $J_{A},\psi$ where $\psi^{2}=[d]$,
see \cite[Theorem 1.2]{BonGee}. The generalized Kummer surface $X$
associated to such $A$ is such that $L_{X}^{2}=0\mod6$. In fact
one can check that the $(1,d)$-polarization they obtain has class
\[
L_{A}=d\g_{1}+\g_{2},
\]
with $L_{A}^{2}=2d=\frac{1}{3}L_{X}^{2}$, and these surfaces form
one of the two families defined in Remark \ref{rem:As-observed-byBarth}.
\end{rem}

\begin{proof}
(Of Theorem \ref{Thm:The-endomorphism-ring}). We proceed as in
the proof of \cite[Theorem 1.2]{BonGee}: the elements $D$ in the
Néron-Severi group of $A$ are identified with alternate forms on
$H_{1}(A,\ZZ)$ (and concretely with alternate matrices in the dual
basis of basis $\a_{1},\b_{1},\a_{2},\b_{2}$).

We know from Theorem \ref{thm:i)-NS(A)} the generators $\d_{1},\d_{2}\in\NS A)$
(which are independent of $n_{1},n_{2},n_{3},n_{4}$) and $D_{A}$
or $L_{A}$ according if $L_{X}^{2}=2$ or $0\mod6$. By \cite[Proposition 5.2.1a]{BirLan},
if $D\in\NS A)$ is such that $D^{2}>0$, then (as an alternate matrix)
it is invertible and for any element $D'\in\NS A)$, the matrix $D^{-1}D'$
is the rational representation of an element of $\End_{\QQ}(A)$ in
basis $\a_{1},\b_{1},\a_{2},\b_{2}$. It is an element of $\End(A)$
if and only if the coefficients are integers. We remark that if $D'\in\NS A)$
is also invertible (as a matrix), then $D'^{-1}D$ is also the rational
representation of an endomorphism. We remark moreover that $\d_{1},\d_{2}$
have determinant $1$ and $\d_{1}^{-1}\d_{2}=j$ (the transpose of
the action of $J_{A}$ on $H_{1}(A,\ZZ)$, since we are in the dual
basis), we define $r=1+2j$, $\phi=r^{-1}\d_{1}^{-1}L_{A}$, and $\psi=\frac{1}{3}\d_{1}^{-1}(L_{A}-\d_{1}-2\d_{2})=\frac{r}{3}(\phi-1)$.

Let us study the endomorphism ring of abelian surfaces $A$ with
polarization
\[
L_{A}=3n_{1}\g_{1}+3n_{2}\g_{2}+3n_{3}\g_{3}+n_{4}(\g_{3}+\g_{4}),
\]
so that $L_{X}^{2}=6(n_{1}n_{2}+n_{3}^{2}+n_{3}n_{4})+2n_{4}^{2}=\frac{1}{3}L_{A}^{2}=2\mod6$.
We already know $j$ and one has 
\[
\phi=\left(\begin{array}{cccc}
2n_{3}+n_{4} & n_{3}+n_{4} & 2n_{2} & n_{2}\\
-n_{3} & -2n_{3}-n_{4} & -n_{2} & -2n_{2}\\
2n_{1} & n_{1} & -2n_{3}-n_{4} & -n_{3}\\
-n_{1} & -2n_{1} & n_{3}+n_{4} & 2n_{3}+n_{4}
\end{array}\right),
\]
\[
\begin{array}{c}
\psi=\frac{1}{3}\left(\begin{array}{cccc}
n_{4}-1 & -3n_{3}-n_{4}-2 & 0 & -3n_{2}\\
-3n_{3}-2n_{4}+2 & -n_{4}+1 & -3n_{2} & 0\\
0 & -3n_{1} & n_{4}-1 & 3n_{3}+2n_{4}-2\\
-3n_{1} & 0 & 3n_{3}+n_{4}+2 & -n_{4}+1
\end{array}\right),\end{array}
\]
both have trace $0$. These formulas are for the elements $\d_{1},\d_{2}$
specified in the proof of Theorem \ref{thm:i)-NS(A)} when $n_{4}=1\mod3$
(the coefficients of $\psi$ are integers). Similar formulas exist
for $n_{4}=2\mod3$. It is easy to check that in both cases $r\phi=-\phi r$
and $\phi^{2}=\frac{1}{2}L_{X}^{2}$, therefore the ring $\QQ[j,\phi]$
is the quaternion algebra $\frac{(-3,\frac{1}{2}L_{X}^{2})}{\QQ}$.
\\
In order to understand the generators of the endomorphism ring $\End(A)$,
let us take $4$ variables $a_{1},\dots,a_{4}\in\QQ$ and search if
the matrix 
\[
a_{1}+a_{1}j+a_{2}\phi+a_{4}\psi
\]
has integral coefficients. That gives $16$ equations for the $4$
unknowns $a_{1},\dots,a_{4}$. The ideal generated by all the $4$
by $4$ minors of that system is $2\la n_{1},\dots,n_{4}\ra^{2}$,
which is equal to $2\ZZ$ since by hypothesis the polarization is
primitive, and therefore the $n_{j}$ are coprime. Thus if $p$ is
a non-trivial denominator of one of the $a_{k}$, then $p=2$. By
reducing the equation modulo $2$, one finds exactly one more class
(modulo the lattice generated by $1,j,\phi,\psi$), which is $\psi'=\frac{1}{2}(1-\phi-\psi)$,
but since one can check that $\psi'=j\psi$, we obtain that $j,\psi$
generate the ring $\End(A)$ and the elements $\beta_{1}=1,\beta_{2}=j,\beta_{3}=\psi,\beta_{4}=j\psi$
form a basis of the free $\ZZ$-module $\End(A)$. The reduced trace
of an element $\beta\in\End(A)$ is $\tfrac{1}{2}$ of the trace of
$\beta$ seen as an element of $M_{4}(\QQ)$ (as an element of a quaternion
algebra, it is equal to $\beta+\bar{\beta}$). One can then compute
the reduced discriminant $D_{\End(A)}=\det(\Tr_{red}((\g_{s}\g_{t})_{1\leq s,t\leq4}))$
of $\End(A)$ and one obtains that $D_{\End(A)}=\frac{1}{2}L_{X}^{2}.$

Suppose that $L_{X}^{2}=0\mod6$. Then 
\[
L_{A}=n_{1}\g_{1}+n_{2}\g_{2}+n_{3}\g_{3}+n_{4}(\g_{3}+\g_{4}),
\]
with $L_{X}^{2}=6(n_{1}n_{2}+n_{3}^{2}+3n_{3}n_{4}+3n_{4}^{2})=3L_{A}^{2}.$
The endomorphism $\phi=\d_{1}^{-1}L_{A}$ is
\[
\phi=\left(\begin{array}{cccc}
n_{4} & -n_{3}-n_{4} & 0 & -n_{2}\\
-n_{3}-2n_{4} & -n_{4} & -n_{2} & 0\\
0 & -n_{1} & n_{4} & n_{3}+2n_{4}\\
-n_{1} & 0 & n_{3}+n_{4} & -n_{4}
\end{array}\right),
\]
which has square $\phi^{2}=\frac{1}{6}L_{X}^{2}$. Moreover $r\phi=-\phi r$.
Using the fact that $n_{1},\dots,n_{4}$ are coprime, one computes
that either $\End(A)$ is the ring generated by $j,\phi$, (with reduced
discriminant $D_{\End(A)}=\frac{1}{2}L_{X}^{2})$, or one can find
another class if $n_{1}=n_{2}=n_{3}=0\mod3$ and $n_{4}\neq0\mod3$.
But by Proposition \ref{PROP:ClasseOfLA} the gcd of $n_{1},n_{2},n_{3}$
must be coprime to $3$. 
\end{proof}
A direct consequence of Theorem \ref{Thm:The-endomorphism-ring} is
\begin{cor}
\label{cor:Two-complex-toriUnik-endo}Two complex tori $A,B$ with
order $3$ automorphism such that $L_{X}^{2}=L_{Y}^{2}$, where $X=\Km_{3}(A),\,Y=\Km_{3}(B)$
have isomorphic endomorphism ring.
\end{cor}

\section{\label{sec:Arithmetic-and-Abelian}Arithmetic and Abelian surfaces
with quaternionic multiplication}

\subsection{\label{subsec:Abelian-Surf-order3}Complex torus with an order $3$
symplectic automorphism}

Let $J_{A}$ be an order $3$ symplectic automorphism of $A$. One
has: 
\[
J_{A}^{2}+J_{A}+I_{A}=0
\]
in $\End(A)$, where $I_{A}$ is the identity of $A$. The Néron-Severi
group of $A$ has rank $\rho_{A}=2,3$ or $4$. From Theorem \ref{Thm:The-endomorphism-ring},
we obtain:
\begin{prop}
\label{prop:AisQM-by-order3}Suppose that $\rho_{A}=3$ and $L_{A}^{2}\neq0$.
The surface $A$ has quaternionic multiplication: there exists a quaternion
algebra $H$ over $\QQ$ and an order $\OO$ in $H$ such that $\OO\simeq\End(A)$.
Since $J_{A}^{2}+J_{A}+I_{A}=0$, the ring of Eisenstein integers
is contained in $\OO$. The quaternion algebra $H$ is indefinite
(i.e. $H\otimes_{\QQ}\RR\simeq M_{2}(\RR)$) if and only if $L_{A}^{2}>0$.
\end{prop}

 Let us therefore study quaternion algebras and their orders $\OO$
which contain the ring $\ZZ[j]$, with $j^{2}+j+1=0$. As we will
see, the Shimura construction gives a converse: if $\OO$ is an order
in a quaternion algebra such that $\ZZ[j]\hookrightarrow\OO$, then
there exists a one dimensional family of abelian surfaces $A$ with
$\OO\simeq\End(A)$ and Picard number $3$ (generically). Let us start
by recalling the definitions for quaternion algebras and by fixing
the notations. 


\subsection{\label{subsec:Quaternion-algebras-and}Quaternion algebras and their
discriminant}

More details on the definitions related to quaternion algebras are
given in the long version v1 of this paper on arXiv. 

\subsubsection{\label{subsec:First-definitions-and}First definitions and notations}

For $a,b\in\QQ$, we denote by $H=\frac{(a,b)}{\QQ}$ the \textit{quaternion
algebra} generated by $\a,\b$ with $\a^{2}=a$, $\b^{2}=b$ and $\a\b=-\b\a.$
 The \textit{reduced discriminant} $D_{H}$ of $H$ is the product
of primes $p$ at which there is a ramification i.e. such that $H_{p}=H\otimes_{\QQ}\QQ_{p}$
is division algebra (by convention we count $-1$ as a prime and $\QQ_{-1}=\RR)$.
There are always an even number of such primes, and two quaternion
algebras are isomorphic if and only if they are ramified at the same
places, which is equivalent to ask that their reduced discriminants
are equal. The quaternion algebra $H$ is a matrix algebra ($H\simeq M_{2}(\QQ)$)
if and only if $D_{H}=1$. \\
If the field $\QQ(\sqrt{d})$ \textit{splits} $H$ (i.e.\,$H\otimes\QQ(\sqrt{d})$
is isomorphic to the matrix algebra $M_{2}(\QQ(\sqrt{d}))$ ), then
there exists $c\in\QQ^{*}$ such that $H=\frac{(c,d)}{\QQ}$ and an
embedding $\Phi:H\hookrightarrow M_{2}(\QQ(\sqrt{d})$ is given by
\begin{equation}
\Phi(x+y\a+z\b+t\a\b)=\left(\begin{array}{cc}
x+y\sqrt{d} & z+t\sqrt{d}\\
c(z-t\sqrt{d}) & x-y\sqrt{d}
\end{array}\right).\label{eq:Phi-representation-Quat}
\end{equation}
The reduced norm and trace of $w\in H$ are then equal to the determinant
and trace of $\Phi(w)$. 

\subsubsection{Eichler orders, Ideals and class numbers}

An \textit{Eichler order} is an order which is the intersection of
two maximal orders. If $D_{H}$ is the discriminant of $H$, the \textit{level}
of $\OO$ is the integer $N$ such that $D_{\OO}=ND_{H}$, where $D_{\OO}$
is the reduced discriminant of $\OO$. By \cite[Proposition 1.54]{AlsinaBayer},
the level of an Eichler order is coprime to $D_{H}$, moreover:
\begin{prop}
(\cite[Corollary 1.58]{AlsinaBayer}). Let $N$ be an integer coprime
to $D_{H}$. Up to conjugation, there exists an unique Eichler order
of level $N$.\\
(\cite[Proposition 1.54]{AlsinaBayer}). Let $\OO$ be an order. Suppose
$D_{\OO}=ND_{H}$ with $N$ a square-free and prime to $D_{H}$ integer.
Then $\OO$ is an Eichler order. 
\end{prop}

Two ideals $I,J$ are said \textit{equivalent on the right} if $I=Jh$
for some $h\in H^{*}$. Let $\OO$ be an order. The classes of equivalent
(on the right) $\OO$-ideals are called the right classes of $\OO$.
There is the same notion for the left. The \textit{class number} $h(\OO)$
of $\OO$-ideals is the number of classes of left $\OO$-ideals; that
number is also equal to the number of classes of right $\OO$-ideals.

\subsubsection{\label{subsec:Legendre,-Kronecker-and}Legendre, Kronecker and Hilbert
symbol}

For the notions of \textit{Legendre and Kronecker symbol}s $\left(\frac{\cdot}{\cdot}\right)$,
and the quadratic reciprocity law, we refer to \cite[Section 1.1.2]{AlsinaBayer}
or \cite{Serre}. For a prime $p$, the \textit{Hilbert symbol} is
the function $(-,-)_{p}:\QQ_{p}^{*}\times\QQ_{p}^{*}\to\{-1,1\}$
defined by 
\[
(a,b)_{p}=\left\{ \begin{array}{ll}
1 & \text{ if }z^{2}=ax^{2}+by^{2}\text{ has a non-trivial solution}\\
-1 & \text{otherwise}.
\end{array}\right.
\]
For the computation of the Hilbert symbol, see \cite{Serre}, or the
version 1 of that paper on arXiv. One has:
\begin{prop}
The quaternion algebra $\frac{(a,b)}{\QQ_{p}}$ is isomorphic to $M_{2}(\QQ_{p})$
if and only if $(a,b)_{p}=1$. 
\end{prop}

Two quaternion algebras $\frac{(a,b)}{\QQ}$, $\frac{(c,d)}{\QQ}$
are isomorphic if and only if they have the same Hilbert symbols:
$(a,b)_{p}=(c,d)_{p}$ for all primes $p$, including $-1$. 

\subsubsection{The discriminant of the quaternion algebras $\frac{(-3,d)}{\protect\QQ}$.}

We recall that, following Conway and Sloane, we take the convention
that $-1$ is a prime. Let us study the quaternion algebras $H$ over
$\QQ$ which contain an element $r$ with $r^{2}=-3$, as in Proposition
\ref{prop:AisQM-by-order3}. Let $d\in\QQ$ such that $H=\frac{(-3,d)}{\QQ}$.
Since for any integer $n\neq0$, we have $\frac{(-3,d)}{\QQ}=\frac{(-3,n^{2}d)}{\QQ}$,
we can suppose that $d=p_{1}\cdot\dots\cdot p_{m}$ is a square free
integer. Let us denote by $r,\phi_{o}$ the generators of $H$ with
$r^{2}=-3,\phi_{o}^{2}=d$ and $r\phi_{o}=-\phi_{o}r$, so that $1,r,\phi_{o},r\phi_{o}$
is a $\QQ$-basis of $H$. If $3$ is among the primes $p_{1},\dots,p_{m}$,
then $(\frac{1}{3}r\phi_{o})^{2}=\frac{d}{3}\in\ZZ$, and $r(\frac{1}{3}r\phi_{o})=-\phi_{o}$,
thus $\frac{(-3,d)}{\QQ}$ is isomorphic to $\frac{(-3,\frac{d}{3})}{\QQ}$,
and we can suppose that $3\nmid d$. Let us define $j=\frac{-1+r}{2}$,
so that $j^{2}+j+1=0$, $r=1+2j\in\ZZ[j]$, and $j\phi_{o}=\phi_{o}j^{2}=\phi_{o}\bar{j}$
(since we know that $\bar{r}\phi_{o}=-r\phi_{o}=\phi_{o}r$). Let
$p$ be a prime dividing $d$ and such that $p=1\mod3$. Then $p$
is a square modulo $3$ and there exists $u\in\ZZ[j]$ with $u\bar{u}=p$.
By writing $u=a+bj$ with $a,b\in\ZZ$, since $\phi_{o}j=\bar{j}\phi_{o}$,
we obtain the relation $\phi_{o}u=\bar{u}\phi_{o}$, and 
\[
(\frac{u}{p}\phi_{o})^{2}=\frac{u\bar{u}}{p^{2}}\phi_{o}^{2}=\frac{p}{p^{2}}d=\frac{d}{p}\in\ZZ,\,\,\,\,r(\frac{u}{p}\phi_{o})=-(\frac{u}{p}\phi_{o})r,
\]
thus the quaternion algebra $H$ is isomorphic to $\frac{(-3,\frac{d}{p})}{\QQ}$.
We can therefore suppose that all prime divisors of $d$ are congruent
to $2$ mod $3$ and $d$ is square free. Let us define $d_{H}=d$:
this is a square-free integer such that the primes dividing it are
congruent to $2$ mod $3$. 
\begin{lem}
Under the above hypotheses on $d_{H}$, the discriminant of $H=\frac{(-3,d_{H})}{\QQ}$
is 
\[
D_{H}=\left\{ \begin{array}{ll}
d_{H} & \text{ if }m\text{ is even }\\
3d_{H} & \text{ if }m\text{ is odd }
\end{array}\right.
\]
where $d_{H}=p_{1}\cdot\dots\cdot p_{m}$ is the decomposition of
$d_{H}$ as a product of primes (including eventually $-1$) $p_{k}$
congruent to $2\mod3$. 
\end{lem}

\begin{proof}
The possible ramification is over primes $3,$ and $p_{1},\dots,p_{m}.$
Let $p\in\{p_{1},\dots,p_{m}\}$ be a natural prime. By Section \ref{subsec:Legendre,-Kronecker-and},
one has $(-3,d_{H})_{p}=(-3,p)_{p}$. Suppose that $p$ is odd; since
$p=2\mod3$, we have
\[
(-3,p)_{p}=\left(\frac{-1}{p}\right)\left(\frac{3}{p}\right)=\left(\frac{p}{3}\right)=-1,
\]
thus $H$ is ramified at $p$. If $p=2$, we also obtain that $(-3,2)_{2}=-1$
and $H$ is ramified at $2$. The quaternion algebra $H$ is definite
if and only if $-1\in\{p_{1},\dots,p_{m}\}$. We have 
\[
(-3,d)_{3}=\prod_{k=1}^{m}(-3,p_{k})_{3}=\prod_{k=1}^{m}\left(\frac{p_{k}}{3}\right)=(-1)^{m},
\]
therefore $3$ is ramified in $H$ if and only if $m$ is odd. 
\end{proof}
\begin{rem}
One has $D_{H}=3d_{H}$ if and only if $d_{H}=2\mod3$, and $D_{H}=d_{H}$
if and only if $d_{H}=1\mod3$. 
\end{rem}

\subsection{\label{subsec:Orders-containing-the-eisen}Orders containing the
Eisenstein integers}

Let us recall that $r,\phi_{o}$ are the generators of $H=\frac{(-3,d_{H})}{\QQ}$
such that $r^{2}=-3,\,\phi_{o}^{2}=d_{H}$, $r\phi_{o}=-\phi_{o}r$.
We also defined $j=\frac{1}{2}(-1+r)$, which satisfies $j^{2}+j+1=0$,
$j\phi_{o}=\phi_{o}\bar{j}=\phi_{o}j^{2}$. If $D_{H}=1\mod3$, let
us define $\t=\frac{r}{3}(\phi_{o}-1)$ (then $\t^{2}=\frac{d_{H}-1}{3}$)
and if $D_{H}=0\mod3$, let us define $\t=\phi_{o}$. 
\begin{prop}
The quaternions $1,j,\t,j\t\in H$ are generators of a maximal order
$\OO_{m}$. \\
Suppose that $H$ is a definite quaternion algebra and let $j'\in\OO_{m}$
be such that $j'^{2}+j'+1=0$. Then $j'=j$ or $j^{2}$.
\end{prop}

\begin{proof}
Suppose that the discriminant of $H=\frac{(-3,d_{H})}{\QQ}$ is $d_{H}$
i.e. $D_{H}=d_{H}=1\mod3$. Let $k\in\NN$ such that $d_{H}=3k+1$.
The quaternions $j$, $\phi_{o}$ and $\t$ are integral over $\ZZ$
and $\OO_{m}$ is a ring. Consider $v_{1}=1,v_{2}=j,v_{3}=\t,v_{4}=j\t$.
The matrix $\left(\Tr(v_{s}v_{t})\right)_{1\leq s,t\leq4}$ is {\small{}
\[
\left(\begin{array}{cccc}
2 & -1 & 0 & 1\\
-1 & -1 & 1 & -1\\
0 & 1 & 2k & -k\\
1 & -1 & -k & 2k+1
\end{array}\right)
\]
}which has determinant $-d^{2}=-D_{H}^{2}$, thus $\cO_{m}$ is maximal
by \cite[Proposition 1.32]{AlsinaBayer}.\\
Suppose that the discriminant of $H=\frac{(-3,d)}{\QQ}$ is $D_{H}=3d_{H}$.
Then $d_{H}=p_{1}\cdot\dots\cdot p_{m}$ is a product of an odd number
of different primes (eventually including $-1$) congruent to $2\mod3$.
The elements $j$ and $\phi_{o}$ are integral over $\ZZ$. Consider
$v_{1}=1,v_{2}=j,v_{3}=\phi_{o},v_{4}=j\phi_{o}$. The matrix $\left(\Tr(v_{s}v_{t})\right)_{1\leq s,t\leq4}$
is {\small{}
\[
\left(\begin{array}{cccc}
2 & -1 & 0 & 0\\
-1 & -1 & 0 & 0\\
0 & 0 & 2d & -d\\
0 & 0 & -d & 2d
\end{array}\right)
\]
}which has determinant $-9d^{2}=-D_{H}^{2}$, thus $\cO_{m}$ is maximal
by \cite[Proposition 1.32]{AlsinaBayer}. 

For the second assertion, we only give a sketch of the proof (see
the long version of that paper on arXiv for a complete proof). The
idea is to write $j'=c_{1}+c_{2}j+c_{3}\t+c_{4}j\t$ for integers
$c_{i}$ and to remark that $j'^{2}=\bar{j'}$, so that $1=j'^{3}=j'\bar{j'}$
is expressed as a quadratic form in the $c_{i}$. The hypothesis on
$H$ (which is in fact $d_{H}<0$) insures that this quadratic form
is definite positive, and finally the only solution for the relation
$j'^{2}+j'+1=0$ to hold is that $j'=j$ or $j'=j^{2}$.
\end{proof}
Let us define the order 
\[
\OO_{\mu}=\ZZ[j]+\ZZ[j]\mu\t\subset\OO_{m},
\]
where $\t$ equals $\frac{r}{3}(\phi_{o}-1)$ or $\phi_{o}$ according
if $D_{H}=1\text{ or }0\mod3$. Let $\OO$ be an order contained
in $\OO_{m}$. 
\begin{prop}
\label{prop:All-Orders}Suppose that $j\in\OO$. There exists $\mu\in\ZZ[j]$
such that $\OO=\OO_{\mu}$ and the discriminant of $\OO_{\mu}$ is
$\mu\bar{\mu}D_{H}$.
\end{prop}

\begin{proof}
Since $\OO$ contains $\ZZ[j]$, this is a rank $2$ $\ZZ[j]$-module
and since $\ZZ[j]$ is principal, we can write $\OO=\ZZ[j]\oplus\ZZ[j]\tau$
for some $\tau$ in the order $\OO_{m}$, which has basis $1,j,\t,j\t$,
and the result follows. The assertion on the discriminant follows
from a direct computation.
\end{proof}
Let $\mu,\mu'\in\ZZ[j]$ be such that $\mu\bar{\mu}=\mu'\bar{\mu'}$. 
\begin{prop}
\label{prop:order-Omu-iso-Omu-prime}The order $\OO_{\mu}$ is isomorphic
to $\OO_{\mu'}$.
\end{prop}

\begin{proof}
The only trouble is with the primes $p_{1},\dots,p_{r}$ dividing
$\mu\bar{\mu}$ which are split in $\ZZ[j]$. Let $\mathfrak{p}_{1},\bar{\mathfrak{p}}_{1},\dots,\mathfrak{p}_{r},\bar{\mathfrak{p}}_{r}$
be the primes in $\ZZ[j]$ over these primes $p_{k}$. There is a
unique factorization (up to invertible elements): 
\[
\mu=r^{a}\mathfrak{p}_{1}^{a_{1}}\bar{\mathfrak{p}}_{1}^{b_{1}}\dots\mathfrak{p}_{r}^{a_{r}}\bar{\mathfrak{p}}_{r}^{b_{r}}m_{2},
\]
for some integers $a,a_{k},b_{k}$, and $m_{2}\in\ZZ$, a product
of primes congruent to $2\mod3$. Consider $\d=\mathfrak{p}_{1}^{b_{1}}\mathfrak{p}_{2}^{b_{2}}\dots\mathfrak{p}_{r}^{b_{r}}$
and define
\[
\mu''=r^{a}\mathfrak{p}_{1}^{a_{1}+b_{1}}\mathfrak{p}_{2}^{a_{2}+b_{2}}\dots\mathfrak{p}_{r}^{a_{r}+b_{r}}m_{2}.
\]
One can check that $\d(\bar{\d}\phi_{o})\d^{-1}=\d\phi_{o}$ (thus
$\d(\mu\phi_{o})\d^{-1}=\mu''\phi_{o}$), that $\d\bar{\d}\t\d^{-1}=\d\t+\g$
for some $\g\in\ZZ[j]$ (thus $\d(\mu\t)\d^{-1}\in\OO_{\mu''}$),
and that therefore the conjugation map $x\mapsto\d x\d^{-1}$ is an
isomorphism $\OO_{\mu}\mapsto\OO_{\mu''}$. We obtain similarly that
$\OO_{\mu'}$ is isomorphic to $\OO_{\mu''}$, since $\mu'=(-j)^{b}r^{a}\mathfrak{p}_{1}^{a_{1}'}\bar{\mathfrak{p}}_{1}^{b_{1}'}\dots\mathfrak{p}_{r}^{a_{r}'}\bar{\mathfrak{p}}_{r}^{b_{r}'}m_{2}$
for some integers $b,a_{k}',b_{k}'$ such that $a_{k}'+b_{k}'=a_{k}+b_{k}$,
and the result follows. 
\end{proof}
Suppose that the quaternion algebra $H/\QQ$ is indefinite. Then all
the maximal orders are conjugated to $\OO_{m}$. Therefore:
\begin{cor}
\label{cor:order-determined-by-discri}An order of $H$ containing
$\ZZ[j]$ is determined up to isomorphism by its discriminant.
\end{cor}

Still in the indefinite case, let $\OO\subset H$ be an order containing
an element $j'$ with $j'^{2}+j'+1=0$. Up to conjugation by an element
of $H$, we can suppose that $\OO\subset\OO_{m}$. The group of Atkin--Lehner
involutions of $\OO_{m}$ acts transitively on the orbits of order
$3$ sub-groups of $\OO_{m}^{*}$. Therefore, up to conjugation and
replacing $j'$ by $j'^{2}$, we can suppose that up to isomorphism
$j'=j$. We thus obtained up to isomorphism all the possible endomorphism
rings of abelian surfaces which have Picard number $3$ and an order
$3$ symplectic automorphism. 
\begin{rem}
Note that using Corollary \ref{cor:Two-complex-toriUnik-endo}, one
obtains more generally that for any endomorphism ring of a complex
torus $A$ with an order $3$ symplectic group, there exists a quaternion
algebra $H$ and $\mu\in\ZZ[j]$ such that $\End(A)=\OO_{\mu}$. Indeed
we will see below that each ring $\OO_{\mu}$ determines such a complex
torus $B$, and the values of $L_{Y}^{2}$ (for $Y=\Km_{3}(B)$) when
$H$ and $\mu$ varies form the set $\{\ell\in\ZZ\,|\,\ell=0\text{ or }2\mod6\}$. 
\end{rem}

\subsection{\label{subsec:Shimura-curves}Shimura curves}

 Let us denote by $\mathcal{H}=\{z\in\CC\,|\,\Im m(z)>0\}$ the Poincaré
upper-half plane and by $\bar{\cH}$ its conjugate. The group $SL_{2}(\RR)$
(mod $\pm I_{2}$) acts on $\mathcal{H}$ (and also on $\bar{\cH}$)
by homographic transformations: 
\[
\left(\g=\left(\begin{array}{cc}
a & b\\
c & d
\end{array}\right),z\right)\in SL_{2}(\RR)\times\mathcal{H}\longrightarrow\g(z)=\frac{az+b}{cz+d}\in\mathcal{H}.
\]
A transformation $\g$ is called \textit{elliptic} if it has a fixed
point $z\in\mathcal{H}$. This is equivalent to require $|\Tr(\g)|<2$,
in particular $\g\neq\pm I_{2}$. 

Let $\G$ be a discrete (for the usual topology) subgroup of $SL_{2}(\RR)$.
A point $z\in\mathcal{H}$ is called elliptic with respect to $\G$
if there exists an elliptic transformation $\g\in\G$ such that $\g(z)=z$.
The isotropy group of a point $z\in\mathcal{H}$ is the group $\G_{z}=\{\g\in\G\,|\,\g(z)=z\}$.
Since $\G$ is discrete, the isotropy group of an elliptic point is
finite and cyclic. By definition, the \textit{order} of an elliptic
point $z$ is $|\G_{z}|$ if $-I_{2}\notin\G$, and $\frac{1}{2}|\G_{z}|$
if $-I_{2}\in\G$. If $z\in\mathcal{H}$ is an elliptic point, then
for any $\g\in\G$, the element $\g(z)\in\mathcal{H}$ is an elliptic
point and their isotropy groups are conjugate: $\G_{\g(z)}=\g\G_{z}\g^{-1}$. 

Let $H=\frac{(c,d)}{\QQ}$ be an indefinite quaternion algebra over
$\QQ$, with $d>0$. Let $I\subset H$ be a $\ZZ$-ideal and let $\OO$
be the left order of $I$. We denote by $\OO^{*}$ the group of invertible
elements and by $\OO^{1}\subset\OO^{*}$ the sub-group of reduced
norm $1$ elements. The group $\OO^{1}$ has index $2$ or $1$ in
$\OO^{*}$ according if there exists an element of $\OO^{*}$ of norm
$-1$ or not. The group $\G(\OO^{1}):=\Phi(\OO^{1})$ (respectively
$\G(\OO^{*}):=\Phi(\OO^{*})$) is a sub-group of $SL_{2}(\QQ(\sqrt{d})),$
which is discrete in $SL_{2}(\RR)$ (where the map $\Phi$ is defined
in Section \ref{subsec:First-definitions-and}). The elements of $\G(\OO^{*})$
are called \textit{quaternionic transformations}. Since $H$ is defined
over $\QQ$, the elliptic transformations of $\G(\OO^{*})$ have order
$2$ or $3$ (see \cite[Lemma 2.10]{AlsinaBayer}).

The group $\G(\OO^{*})$ (respectively $\G(\OO^{1})$ acts on $\mathcal{H}\cup\bar{\cH}$
(respectively $\cH$ and $\bar{\cH}$). If it exists, an element $\g\in\OO^{*}$
of norm $-1$ exchanges the upper and lower half-planes $\cH$ and
$\bar{\cH}$. The quotient $\cH/\G(\OO^{1})$ is a complex Riemann
surface. The quotient $\mathcal{X}(\OO)=(\mathcal{H}\cup\bar{\cH})/\G(\OO^{*})$
is isomorphic to $\cH/\G(\OO^{1})$ if and only if $-1$ is a norm.
If $-1$ is not a norm, then $\mathcal{X}(\OO)$ is the disjoint union
of the two curves $\cH/\G(\OO^{1})$, $\bar{\cH}/\G(\OO^{1})$. The
compactification of $\mathcal{X}(\OO)$ is a Shimura curve. The curve
$\mathcal{H}/\G(\OO^{1})$ is compact if and only if $H$ is a division
algebra; we will mainly suppose that this is the case in the following.
A point $p$ on the curve $\mathcal{H}/\G(\OO^{1})$ is called elliptic
if this is the image of some elliptic point $z$ for $\G(\OO^{1})$
by the quotient map $\mathcal{H}\to\mathcal{H}/\G(\OO^{1})$; by definition,
its order is the order of the elliptic point $z$. We remark that
$\Phi(-1)=-I_{2}\in\G(\OO^{1})$ and we obtain immediately:
\begin{prop}
\label{prop:Bijection-oder3elliptic-gpe}There is a bijection between
the set of elliptic points of order $k$ on $\mathcal{H}/\G(\OO^{1})$
and the set of orbits under conjugation by $\G(\OO^{1})$ of order
$2k$ subgroups of $\G(\OO^{1})$.
\end{prop}

The curve $\mathcal{X}(\OO)=(\mathcal{H}\cup\bar{\cH})/\G(\OO^{*})$
is a coarse moduli space for abelian surfaces with quaternionic multiplication
by $\OO$. 

Let $e_{k}(\OO)$ be the number of elliptic points of order $k$ ($k\in\{2,3\}$)
of the group $\G(\OO^{1})$. Suppose that $\OO$ is an Eichler order
of level $N$ (necessarily coprime to $D_{H}$) or is maximal $(N=1)$,
then
\begin{prop}
\label{prop:-The-numberellipticPoints}\cite[Proposition 2.29]{AlsinaBayer}
The number of order $3$ elliptic points of $\mathcal{H}/\G(\OO^{1})$
is{\small{}
\[
\begin{array}{c}
e_{3}(\OO)=\left\{ \begin{array}{cc}
\prod_{p|D_{H}}(1-\left(\frac{-3}{p}\right))\prod_{p|N}(1+\left(\frac{-3}{p}\right)) & \text{if}\,\,9\nmid N\\
0 & \text{if}\,\,9|N,
\end{array}\right.\end{array}
\]
}where $D_{H}$ is the discriminant of $H$.
\end{prop}

This number is $0$ or a power of $2$. One has $e_{3}(\OO)=1$ if
and only if $D_{H}=1$ and $N=1$ or $3$.  We remark that as soon
as the level $N$ is divisible by a prime $p$ such that $p=2\mod3$,
one has $e_{3}(\OO)=0$, which means that there are no injective ring
homomorphism $\ZZ[j]\to\OO$, (where $j^{2}+j+1=0$).

\subsection{\label{subsec:Abel-surf-with-QM}Complex $2$-tori with QM and generalized
Kummer surfaces}

Let $H=\frac{(c,d)}{\QQ}$ be a quaternion algebra over $\QQ$ (definite
or not) and let $\OO\subset H$ be an order. One says that a complex
$2$-torus $A$ admits a quaternionic multiplication by $H$ if there
is an embedding 
\[
\iota:H\hookrightarrow\End(A)\otimes\QQ,
\]
moreover, $A$ admits a quaternionic multiplication by the order $\OO$
if 
\[
\iota(H)\cap\End(A)=\iota(\OO)
\]
in $\End(A)\otimes\QQ$ (then, for the general torus, one has $\End(A)=\iota(\OO)$).
Let $\Phi:H\hookrightarrow M_{2}(\CC)$ or $M_{2}(\RR)$ be an embedding
(see Section \ref{subsec:First-definitions-and}), according if $H$
is definite or not. The complex $2$-tori with quaternionic multiplication
by $\OO$ are those of the form 
\[
A_{u,v}=\CC^{2}/\Phi(I)\left(\begin{array}{c}
z\\
z'
\end{array}\right),
\]
where $I\subset\OO$ is a left $\OO$-ideal, for $z,z'\in\CC$ with
$zz'\neq0$ if $H$ is definite, and if $H$ is indefinite, $z'=1$
and $z$ is in $\cH\cup\bar{\cH}$; in that case we denote simply:
$A_{z}=A_{u,v}$. These results are due to Shimura \cite{Shimura}
for $H$ indefinite, and for $H$ definite, we refer to the paper
of Shimizu \cite[Section 4]{Shimizu}. Note that in the definite
case, if $\l\in\CC^{*}$, one has $A_{\l u,\l v}\simeq A_{u,v}$ and
the complex $2$-tori with quaternionic multiplication by $\OO$ are
therefore parametrized by $\PP^{1}$. 

When $H$ is indefinite, one has $A_{z}\simeq A_{w}$ if and only
if $z=\g(w)$ for some $\g\in\OO^{*}$, where $\OO^{*}=\{\g\in\OO\,|\,\Nr(\g)=\pm1\}$
is the group of invertible elements of $\OO$. 

Two left $\OO$-ideals $I,J$ give the same family of complex $2$-tori
with endomorphism ring $\OO$ if and only if there exists a principal
ideal $P$ such that $I=JP$. The automorphism group of the general
complex $2$-tori $A$ is isomorphic to $\OO^{*}$.  

\section{\label{sec:Junction-between-Shimura}Junction between Shimura and
Barth constructions}

\subsection{\label{subsec:The-polarization-on}The polarization on the generalized
Kummer surface and its isogeny classes}

As in Section \ref{subsec:Abel-surf-with-QM}, let $d_{H}$ be a square
free integer divisible only by primes congruent to $2\mod3$ and let
$H$ be the quaternion algebra $H=\frac{(-3,d_{H})}{\QQ}$, we denote
by $r,\phi_{o}$, the generators such that $r^{2}=-3$, $\phi_{o}^{2}=d_{H}$,
$\phi_{o}r=-r\phi_{o}$. The discriminant of $H$ satisfies $D_{H}=d_{H}$
if $d_{H}=1\mod3$ and $D_{H}=3d_{H}$ if $d_{H}=2\mod3$. Let us
also define $j=\frac{1}{2}(-1+r)$. We will suppose in this section
that $H$ is indefinite. 
\begin{defn}
For an integer $c$, we denote by $\rad_{2}(c)$ the product of the
primes congruent to $2\mod3$ that divide $c$ to an odd power (for
example $\text{rad}_{2}(12)=1$).
\end{defn}

Let us say that two generalized Kummer surfaces $X=\Km_{3}(A),\,Y=\Km_{3}(B)$
are in the same isogeny class if there exists an isogeny of complex
tori $A\to B$. If this is the case, there exists a natural dominant
rational map $X\dasharrow Y$. 

Let $A$ be an abelian surface with an order $3$ automorphism $J_{A}$,
with Picard number $3$ and quaternionic multiplication by $H$. The
aim of this section is to prove the following result:
\begin{thm}
\label{thm:Thm35}There exists $\mu\in\ZZ[j]$ such that endomorphism
ring of $A$ is (isomorphic to)
\[
\OO_{\mu}=\ZZ[j]+\ZZ[j]\mu\t,
\]
where $\t=\frac{r}{3}(\phi_{o}-1)$ or $\t=\phi_{o}$, according if
$D_{H}=1$ or $0\mod3$. The generalized Kummer surface $X$ is polarized
by $L_{X}$ with $L_{X}^{2}=2\mu\bar{\mu}D_{H}$. The integer $L_{X}^{2}$
determines $\End(A)$. \\
Let $X=\Km_{3}(A),Y=\Km_{3}(B)$ be two generalized Kummer surfaces.
If the surfaces $X$ and $Y$ are in the same isogeny class then $\rad_{2}(\frac{1}{2}L_{X}^{2})=\rad_{2}(\frac{1}{2}L_{Y}^{2})$.
\\
Suppose that $L_{X}^{2}=L_{Y}^{2}$. Then $\End(A)\simeq\End(B)$. 
\end{thm}

Let us recall the following result of Theorem \ref{Thm:The-endomorphism-ring}
according to the cases $L_{X}^{2}=2$ or $0\mod6$:

\begin{tabular}{|l|c|cc|c|c|}
\hline 
 &  &  &  & $\End(A)$ & $D_{\End(A)}$\tabularnewline
\hline 
i) & $L_{X}^{2}=2\mod6$ & $\phi^{2}=\frac{1}{2}L_{X}^{2}=\frac{1}{6}L_{A}^{2},$ & $\psi=\frac{r}{3}(\phi-1)$ & $\ZZ[j,\psi]$ & $\tfrac{1}{2}L_{X}^{2}=\frac{1}{6}L_{A}^{2}$\tabularnewline
\hline 
ii) & $L_{X}^{2}=0\mod6$ & $\phi^{2}=\frac{1}{6}L_{X}^{2}=\frac{1}{2}L_{A}^{2}$ &  & $\ZZ[j,\phi]$ & $\tfrac{1}{2}L_{X}^{2}=\frac{3}{2}L_{A}^{2}$\tabularnewline
\hline 
\end{tabular}

where $D_{\End(A)}$ is the discriminant of $\End(A)$. For proving
Theorem \ref{thm:Thm35}, we need the following result:
\begin{lem}
\label{lem:RAD2}One has $\rad_{2}(D_{\End(A)})=\rad_{2}(\frac{1}{2}L_{X}^{2})=d_{H}$.
\end{lem}

\begin{proof}
By Theorem \ref{Thm:The-endomorphism-ring}, the discriminant of $\End(A)$
equals $\tfrac{1}{2}L_{X}^{2}$. According to Proposition \ref{prop:All-Orders},
there exists $\mu\in\ZZ[j]$ such that $\End(A)\simeq\OO_{\mu}$;
the discriminant of $\OO_{\mu}$ is $\mu\bar{\mu}D_{H}$. Let $p\in\ZZ$
be a prime. If $p=1\mod3$, then $p$ is split in $\ZZ[j]$: there
exists a prime $\mathfrak{p}$ in $\ZZ[j]$ such that $p=\mathfrak{p}\bar{\mathfrak{p}}$.
If $p=2\mod3$, then $p$ is inert: it is still a prime in $\ZZ[j]$.
Finally $3$ is ramified: $3=-r^{2}=r\bar{r}$. Thus for any $\mu\in\ZZ[j]$,
there exists uniquely determined integers $a,m_{1},m_{2}$ such that
$\mu\bar{\mu}=3^{a}m_{1}(m_{2})^{2}$, where $a\geq0$ and the integer
$m_{i}$ ($i\in\{1,2\}$) is a product of primes congruent to $i\mod3$.
Therefore $\text{rad}_{2}(\mu\bar{\mu}D_{H})=d_{H}$ since $d_{H}$
is square free and the primes dividing $d_{H}$ are congruent to $2\mod3$. 
\end{proof}
\begin{proof}
(Of Theorem \ref{thm:Thm35}). Let $X=\Km_{3}(A),\,Y=\Km_{3}(B)$
be two generalized algebraic Kummer surfaces. There exists $\mu,\mu'\in\ZZ[j]$
such that the discriminants of the endomorphism rings of $A,B$ are:
\[
D_{\End(A)}=\mu\bar{\mu}D_{H},\,\,D_{\End(B)}=\mu'\bar{\mu'}D_{H'}.
\]
Suppose that $L_{X}^{2}=L_{Y}^{2}$. Then according to Theorem \ref{Thm:The-endomorphism-ring},
\[
\mu\bar{\mu}D_{H}=\mu'\bar{\mu'}D_{H'},
\]
thus by Lemma \ref{lem:RAD2}, $d_{H}=d_{H'}$ which implies $H\simeq H'$,
and moreover $\mu\bar{\mu}=\mu'\bar{\mu'}$. The orders $\End(A),\End(B)$
being of the same discriminant and of the form $\OO_{\mu}$ for some
$\mu\in\ZZ[j]$, they are isomorphic by Proposition \ref{cor:order-determined-by-discri}. 

Finally, we remark that if the surfaces $X$ and $Y$ are isogeneous,
the abelian surfaces $A$ and $B$ have quaternionic multiplication
by the same quaternion algebra $H$, and therefore $\text{rad}_{2}(\frac{1}{2}L_{X}^{2})=\text{rad}_{2}(\frac{1}{2}L_{Y}^{2})$. 
\end{proof}

\subsection{Order that contains an element of norm $-1$}

Let $\cO$ be the order of a quaternion algebra containing the Eisenstein
integers. Let $A$ be a general abelian surface with quaternionic
multiplication by $\cO$, and let $J$ be an order $3$ symplectic
automorphism of $A$. We denote by $X=\Km_{3}(A,J)$ the associated
order $3$ Kummer surface. For future use in the proof of Proposition
\ref{THEOREM:NumberEllPts}, let us prove the following result:
\begin{lem}
\label{lem:ElementNormm1}The group $\OO^{*}$ contains an element
of norm $-1$ if and only if $L_{X}^{2}=2\mod6$ or $3||L_{X}^{2}$. 
\end{lem}

\begin{proof}
We follow the proof of \cite[Lemma 3.3]{KohelVerril} given for Eichler
orders. Let $Q$ be quadratic form $Q(q)=\Tr(q)^{2}-\Nr(q)$ for $q\in\OO=\OO_{\mu}$.
It is degenerate of rank $3$ and induces a non-degenerate ternary
quadratic form (also denoted by $Q$) on $\OO_{\mu}/\ZZ$, called
the discriminant form. For any $q\in\OO_{\mu}$, the integer $Q(q)$
is the discriminant of the quadratic subring $\ZZ[q]$. It suffices
to show that the discriminant form represents a prime $p$ congruent
to $1\mod4$, since then $\cO=\OO_{\mu}$ contains a real quadratic
order of discriminant $p$, whose fundamental unit has norm $-1$.
\\
Suppose that $d_{H}=2\mod3$ i.e. $D_{H}=3d_{H}$. A basis of $\OO_{\mu}$
is $1,j,\mu\phi,\mu j\phi$ with $\phi^{2}=d_{H}$. Consider $d=\mu\bar{\mu}d_{H}=(\mu\phi)^{2}$.
The discriminant form is 
\[
Q(xj+y\mu\phi+zj\mu\phi)=-3x^{2}+4d(y^{2}+yz+z^{2}).
\]
Suppose that $d_{H}=1\mod3$ i.e. $D_{H}=d_{H}$ is coprime to $3$.
A basis of $\OO_{\mu}$ is $1,j,\mu\psi,\mu j\psi$ with $\psi=\frac{r}{3}(\phi-1)$,
and $\phi^{2}=d_{H}$. Consider $d=\mu\bar{\mu}d_{H}=(\mu\phi)^{2}$.
The discriminant form is 
\[
Q(xj+y\mu\psi+zj\mu\psi)=-3x^{2}+4xy-2xz+z^{2}+4d(y^{2}+yz+z^{2}).
\]
In both cases, this is an indefinite quadratic form. Suppose that
$\mu$ is coprime to $3$. Thus $d$ is coprime to $3$, the form
$Q$ is not $0$ modulo any prime and we conclude as in \cite[Lemma 3.3]{KohelVerril}
that the quadratic form represents prime numbers $p$ congruent to
$1\mod4$, and therefore there exists a unit of norm $-1$. \\
Suppose that $9|L_{X}^{2}$ and consider $\tau=x+yj+u\phi+vj\phi$.
The equation 
\[
\tau\bar{\tau}=x^{2}-xy+y^{2}-3d'(u^{2}-uv+v^{2})=-1
\]
has no solution $\tau=x+yj+u\phi+vj\phi$ modulo $3$ for any $d'$,
thus the equation $\tau\bar{\tau}=-1$ has no solution: there is no
unit of norm $-1$.
\end{proof}

\subsection{\label{subsec:Number-of-generalized}Number of generalized Kummer
structures}

In this section, we consider an indefinite quaternion algebra $H=\frac{(-3,d_{H})}{\QQ}$,
where $d_{H}>0$ is a square-free product of primes congruent to $2\mod3$,
so that the discriminant of $H$ equals to $d_{H}$ or $3d_{H}$.
The generators are $r,\phi_{0}$ with $r^{2}=-3,\phi_{o}^{2}=d_{H}$
and $r\phi_{o}=-\phi_{o}r$. As before, let $\t$ be the quaternion
$\phi_{o}$ or $\frac{1}{3}r(\phi_{o}-1)$ according if $3|D_{H}$
or not. Let $\OO$ be an order containing the sub-ring $\ZZ[j]$.
By Proposition \ref{prop:All-Orders}, we can suppose that $\OO$
is contained in the maximal order $\OO_{m}$, and there exists $\mu\in\ZZ[j]$
such that $\OO$ is the ring
\[
\OO_{\mu}=\ZZ[j]\oplus\ZZ[j]\mu\t.
\]

Let $A=\CC^{2}/\Phi(I)\,^{t}(z,1)$ be an abelian surface with quaternionic
multiplication by $\OO=\OO_{\mu}$ and Picard number $3$ ($I$ is
an $\OO$-ideal). Let $J_{A}$ be the order $3$ automorphism acting
on $A$ corresponding to $j\in\OO$. Let 
\[
X=\Km_{3}(A)=\Km_{3}(A,J_{A})
\]
be the associated generalized Kummer surface: it is the minimal resolution
of $A$ by the group generated by $J_{A}$. 

Let $\mathcal{K}(X)$ denote the set of generalized Kummer structures
on $X$, which is the set of isomorphisms classes of pairs $(B,G)$
such that $X\simeq\Km_{3}(B,G)$, where the pair $(B,G)$ is said
isomorphic to $(B',G')$ if there exists an isomorphism $\phi:B\to B'$
of abelian surfaces such that $G=\phi G'\phi^{-1}$ (see \cite{KRS}). 

Let $\mathcal{C}_{A}$ be the set of orbits (under conjugation by
$\aut(A)$) of order $3$ symplectic automorphism sub-groups $G\subset\aut(A)=\OO^{*}$. 

Let us recall that we proved in Theorem \ref{THEOREM:bidule} that
\[
\FM(A)=\FM(A,\NS A))=\{A,\hat{A}\}
\]
and that by Theorem \ref{thm:Fourier-Mukai}, if $(B,G_{B})\in\cK(X)$,
then $B\in\FM(A)$. 

We also recall that when $L_{X}^{2}\neq0,36\mod54$, by Theorem \ref{thm:Fourier-Mukai},
the generalized Kummer surface $\Km(A,G)$ associated to any group
$G\in\cC_{A}$ is isomorphic to $X=\Km_{3}(A,J_{A})$. Hence in that
case, there is a well-defined map 
\[
\varDelta_{A}:\mathcal{C}_{A}\to\mathcal{K}(X),\,G\to(A,G),
\]
which is moreover injective. Also from the discussion $\cK(X)=\Delta_{A}(\cC_{A})\cup\Delta_{\hat{A}}(\cC_{\hat{A}})$,
with $\Delta_{A}(\cC_{A})=\Delta_{\hat{A}}(\cC_{\hat{A}})$ if $A$
has a principal polarization and with $\Delta_{A}(\cC_{A})\cap\Delta_{\hat{A}}(\cC_{\hat{A}})=\emptyset$
otherwise.

If $L_{X}^{2}=0\text{ or }36\mod54$, then $\cK(X)$ is bounded from
above by the order of $\cC_{A}\cup\cC_{\hat{A}}$, and it can happen
that two pairs $(A,G)$, $(A,G')$ for $G,G'\in\cC_{A}$ give non-isomorphic
K3 surfaces \cite[Corollary 24]{RS4}. 

Let us recall that we denoted by $e_{3}(\OO)$ the number of elliptic
points of order $3$ on the Shimura curve $\cH/\G(\OO^{1})$. This
is also the number of orbits of order $3$ symplectic automorphism
sub-groups $G\subset\aut(A)$ under conjugation by $\OO^{1}\subset\aut(A)$.
We have 
\begin{thm}
\label{THEOREM:NumberEllPts}Suppose that $D_{H}=1$ or that $L_{X}^{2}=0\text{ or }36\mod54$.
Then
\begin{itemize}
\item $|\cK(X)|\leq e_{3}(\OO)$ if \textup{$L_{X}^{2}=2\mod6\text{ or }3||L_{X}^{2}$,}
\item $|\cK(X)|\leq2e_{3}(\OO)$ \textup{otherwise.}
\end{itemize}
Suppose that $D_{H}\neq1$ and $L_{X}^{2}\neq0,36\mod54$. Then 
\begin{itemize}
\item $|\cK(X)|=\frac{1}{2}e_{3}(\OO)\text{ if }L_{X}^{2}=2\mod6\text{ or }3||L_{X}^{2},$
\item $|\cK(X)|=2e_{3}(\OO)$ otherwise.
\end{itemize}
\end{thm}

\begin{proof}
Let $(B,G)\in\cK(X)$ be a Kummer structure on $X$. By Theorem \ref{thm:Fourier-Mukai},
we can suppose that $B=A$ or $B=\hat{A}$.

The surface $A$ is isomorphic to $\hat{\ensuremath{A}}$ if and only
if $A$ is principally polarized, which is equivalent by \ref{THEOREM:bidule}
to 
\[
L_{X}^{2}=2\mod6\text{ or }3||L_{X}^{2}.
\]
By Lemma \ref{lem:ElementNormm1}, this is also equivalent to the
condition that group $\OO^{*}$ contains an element of norm $-1$,
which is equivalent to suppose that $\G(\cO^{1})$ has index $2$
in $\G(\cO^{*})$. 

$\bullet$ Case $L_{X}^{2}=2\mod6\text{ or }3||L_{X}^{2}.$ Let us
suppose that $L_{X}^{2}=2\mod6\text{ or }3||L_{X}^{2}.$ Then the
elements of norm $-1$ exchange the half planes $\cH$ and $\bar{\cH}$
and one has 
\[
\cH\cup\bar{\cH}/\G(\cO^{*})=\cH/\G(\cO^{1}).
\]
In any case, that proves that $\cK(X)$ has order $\leq e_{3}(\cO)$.
Suppose now that $H$ is a skew-field i.e. $D_{H}\neq1$. Let $\g\in\cO^{*}$
be a norm $-1$ element.  The matrix $\Phi(\g)$ has determinant
$-1$ and acts anti-holomorphically on the Poincaré upper-plane by
$z\to(a\bar{z}+b)/(c\bar{z}+d)$. That action induces a real structure
on the Shimura curve $\cH/\G(\OO^{1})$. It is known that this action
has no fixed-point since $H$ is a skew field (see \cite[Section 2.2]{Elkies};
a Shimura curve has no real points), since we suppose $D_{H}\neq1$.
Therefore there are no fixed points among the $e_{3}(\OO)$ order
$3$ elliptic points on the Shimura curve $\mathcal{H}/\G(O^{1})$.
At the level of orbits in $\mathcal{H}$, if $G\subset\OO^{1}$ has
order $3$ and fixes $z\in\mathcal{H}$, then the point $\g\bar{z}\in\mathcal{H}$
is fixed by $\g G\gamma^{-1}\subset\OO^{1}$, and is not in the orbit
of $z$ under $\cO^{1}$. Therefore $\g$ maps the orbit $O(G)=\{\tau G\tau^{-1}\,|\tau\in\OO^{1}\}$
to the orbit $O(\g G\g^{-1})$, which is also equal to $\g O(G)\g^{-1}$.
The action of the conjugation by $\g$ on the set of $e_{3}(\OO)$
orbits $O(G)=\{\tau G\tau^{-1}\,|\tau\in\OO^{1}\}$ has therefore
no fixed element, hence the number of conjugacy classes of order $3$
groups contained in $\cO^{*}$ under the action by conjugation of
$\cO^{*}$ is $\tfrac{1}{2}e_{3}(\OO)$.

$\bullet$ Suppose that $L_{X}^{2}=18\mod54$. From Theorems \ref{thm:Fourier-Mukai}
and \ref{THEOREM:bidule}, the Kummer structures on $X$ are the pairs
$(A,G)$, $(\hat{A},\hat{G})$ with $G,\hat{G}$ symplectic of order
$3$. One has $\cO^{*}=\cO^{1}$. The abelian surface $A$ has $e_{3}(\OO)$
orbits of symplectic order $3$ automorphism groups, and so has the
surface $\hat{A}\not\simeq A$. Thus the number of Kummer structures
on $X$ is $2e_{3}(\cO)$.

$\bullet$ Suppose that $L_{X}^{2}=6k$ with $k=0$ or $6\mod9$.
Since all Kummer structures on $X=\Km{}_{3}(G)$ are of the form $(A,G)\text{ or }(\hat{A},\hat{G})$
with $G$ of order $3$, we obtain the upper-bound $2e_{3}(\cO)$.
This is only an upper bound since it can happen that there is a group
$G$ such that $(A,G)$ is not a Kummer structure of $X$. 
\end{proof}
 Suppose that $\mu=\mathfrak{q}_{1}\dots\text{\ensuremath{\mathfrak{q}}}_{n}\in\ZZ[j]$
is a product of primes $\mathfrak{q}_{1},\dots\mathfrak{,q}_{n}$
in $\ZZ[j]$, with norm some natural primes $q_{1},\dots,q_{n}$ such
that $q_{j}=1\mod3$ and the $q_{j}$ are distinct. Let $A$ be a
general abelian surface with quaternionic multiplication by $\OO_{\mu}$
and let $X=\Km_{3}(A)$ be the associated generalized Kummer surface.
A consequence of Propositions \ref{prop:-The-numberellipticPoints}
is 
\begin{cor}
\label{cor:pleindestructures}The order $\OO_{\mu}=\ZZ[j]\oplus\ZZ[j]\mu\t$
is an Eichler order of level $\mu\bar{\mu}$. The number of Kummer
structures on $X$ is equal to $2^{m+\e},$ where $m$ is the number
of primes dividing $\frac{1}{2}L_{X}^{2}=\mu\bar{\mu}D_{H}$, and
where $\e=-2$ if $3|D_{H}$, $\e=-1$ otherwise.
\end{cor}

\begin{proof}
The order $\OO_{\mu}$ is generated by $\g_{1}=1,\g_{2}=j,\g_{3}=\mu\t,\g_{4}=j\mu\t$,
and one computes that the reduced discriminant of $\OO_{\mu}$ satisfies
\[
D_{\OO_{\mu}}^{2}=\det(\Tr(\g_{i}\g_{j}))=(\mu\bar{\mu})^{2}D_{H}^{2}.
\]
By the hypothesis on $\mu,$ the integer $\mu\bar{\mu}$ is square
free, thus by \cite[Proposition 1.54]{AlsinaBayer}, $\OO_{\mu}$
is an Eichler order of level $\mu\bar{\mu}$. 

The formula in Proposition \ref{prop:-The-numberellipticPoints} gives
that $e_{3}(\OO_{\mu})=2^{m}$ if $D_{H}$ is coprime to $3$ and
$e_{3}(\OO_{\mu})=2^{m-1}$ if $3|D_{H}$. With the hypothesis on
$H$ and $\mu$, either $L_{X}^{2}=2\mod6$ or $3||L_{X}^{2}$ and
Theorem \ref{THEOREM:NumberEllPts}, gives that the number of Kummer
structures equals to $2^{m+\e}$. 
\end{proof}
\begin{rem}
From Corollary \ref{cor:pleindestructures} and the above examples
of Eichler orders, we see that in the same isogeny class of $X$,
the number of generalized Kummer structures can be arbitrarily large.
Moreover letting varying $H$, and taking the unique maximal order
(up to conjugation), we also obtain generalized Kummer surfaces with
an arbitrarily large number of Kummer structures.
\end{rem}

Using the formulas in \cite{CR}, the following Table gives the number
of Kummer structures for some low values of $L_{X}^{2}$:

\vspace{1mm}

\begin{tabular}{|c|c|c|c|c|c|c|c|c|c|c|c|c|c|}
\hline 
$L_{X}^{2}$ & 6 & 12 & 18 & 24 & 30 & 36 & 42 & 48 & 54 & 60 & 66 & 72 & 90\tabularnewline
\hline 
$|\mathcal{K}(X)|$ & $1$ & $1$ & $1$ & $1$ & $1$ & $4^{*}$ & $1$ & $1$ & $3^{*}$ & $2$ & $1$ & $2$ & $4^{*}$\tabularnewline
\hline 
\end{tabular}

where for $L_{X}^{2}=0$ or $36\mod54$, the exponent $^{*}$ indicates
that this is only an upper bound.

\section{\label{subsec:Classification-of-the}Moduli spaces of generalized
Kummer surfaces}

\subsection{The algebraic case: irreducibility of the moduli space of generalized
Kummer surfaces}

For each integer $\ell>0$ such that $\ell=0$ or $2\mod6$, let $\mathcal{NS}_{\ell}$
be the Néron-Severi lattice of the generalized Kummer surface $X$
with Picard number $19$ and the divisor $L_{X}$ such that $L_{X}^{2}=\ell$.
We denote by $\cM_{\ell}$ the one dimensional moduli space of algebraic
K3 surfaces $X$ polarized by $\mathcal{NS}_{\ell},$ as defined in
\cite{Dolgachev}.  Let us recall that we obtain in Section \ref{subsec:The-N=0000E9ron-Severi-latofA}
a unique rank $3$ lattice $\text{NS}_{\ell}$ such that, for $X=\Km_{3}(A)$
with $L_{X}^{2}=\ell$, one has $\NS A)\simeq\text{NS}_{\ell}$. Let
us denote by $\OO=\OO_{\mu}$ the unique quaternion order with discriminant
$\bar{\mu}\mu D_{H}=\tfrac{1}{2}\ell$, so that the abelian surface
$A$ has quaternionic multiplication by $\OO$. 

In \cite{Barth} Barth asks how many irreducible components has the
moduli space $\cM_{\ell}$. Such a question on algebraic lattice polarized
K3 surfaces has been studied for example in \cite{Dolgachev}. However,
since the Néron-Severi group of the generalized Kummer surfaces is
large, the results of \cite{Dolgachev} do not apply in our situation.
We obtain:
\begin{thm}
\label{thm:une-ou-deux-compo-irred} The order $\OO$ is principal.
The moduli space of abelian surfaces with quaternionic multiplication
by $\OO$ has one irreducible component if $\ell=2\mod6$ or $3||\ell$,
and $2$ irreducible components otherwise. The moduli space $\cM_{\ell}$
of generalized Kummer surfaces is irreducible for any $\ell>0$ (with
$\ell=0$ or $2\mod6$). 
\end{thm}

The remaining of this section is devoted to the proof of Theorem \ref{thm:une-ou-deux-compo-irred},
during which we give the definition of what we mean by the moduli
space of abelian surfaces with quaternionic multiplication by $\OO$.

A Hodge structure of weight $2$ and K3 type on the lattice $U^{\oplus3}$
is the data of spaces $V^{p,q}\subset U^{\oplus3}\otimes\CC$ for
$p+q=2,\,p\geq0$ such that $\overline{V^{p,q}}=V^{q,p}$ and $\dim_{\CC}V^{2,0}=1$,
with a generator $\o$ of $V^{2,0}$ satisfying $\o.\o=0,\,\o.\bar{\o}>0$,
where the intersection form is induced from the intersection form
on $U^{\oplus3}$. For a lattice $\text{NS}_{\ell}\subset U^{\oplus3}$
of rank $3$ and signature $(1,2)$, the family of Hodge structures
$(V^{p,q}){}_{p+q=2}$ such that $V^{1,1}\cap U^{\oplus3}=\text{NS}_{\ell}$
is one-dimensional (it is biholomorphic to $\mathcal{H}\cup\bar{\mathcal{H}}$,
where $\mathcal{H}$ is the upper half plane, see e.g. \cite[Section 2.3]{Barth}
or \cite[Section 7.2.3]{Voisin}).

Let $\text{NS}$ be an even rank $3$ lattice of signature $(1,2)$
or $(0,3)$. The families of complex tori $A$ with $\NS A)\simeq\text{NS}$
are in one-to-one correspondence with the embeddings $\text{NS}\hookrightarrow U^{\oplus3}$
up to isometry. By \cite[Chapter 14, Proposition 1.8]{Huyb}, such
an embedding always exists. The orthogonal complement $T$ of $\text{NS}$
in $U^{\oplus3}$ is such that $T(-1)$ is in the same genus as $\text{NS}$.
Conversely any lattice $T$ in the genus of $\text{NS}(-1)$ can be
glued to $\text{NS}$ using a glue map $\varphi:\mathrm{A}_{\text{NS}}\to\mathrm{A}_{T}$
to form a rank $6$ even unimodular lattice $\text{NS}\oplus_{\varphi}T$,
which therefore is isomorphic to $U^{\oplus3}$ (for glueing theory
of lattices, we refer to the detailed Section 2 of \cite{MCM}). 

Let us denote by $\text{T}_{\ell}$ the transcendental lattice of
an abelian surface $A$ with $\NS A)\simeq\text{NS}_{\ell}$. Since
$\ell>0$, by Corollary \ref{cor:Bidule2}, the lattice $\text{T}_{\ell}$
is isomorphic to $\text{NS}_{\ell}(-1)$. By Theorem \ref{THEOREM:bidule},
the co-set $\mathcal{P}=O(\text{T}_{\ell})\setminus O(\mathrm{A}_{\text{T}_{\ell}})$
has a unique element. Thus (see \cite[Definition 1]{HLOY2}), if $\iota_{0},\iota_{1}:\text{T}_{\ell}\hookrightarrow U^{\oplus3}$
are any two embeddings, there exist $g\in O(\text{T}_{\ell}),$ $\Phi\in O(U^{\oplus3})$
such that $\iota_{1}=\Phi\circ\iota_{0}\circ g$. 

Let us prove that the glueing of $\text{NS}_{\ell}$ and $\text{T}_{\ell}$
is unique (one could refer to the general result of \cite[Theorem]{MirMor1},
but we prefer to give a proof here). Let $\varphi_{1},\varphi_{2}:\mathrm{A}_{\text{T}_{\ell}}\to\mathrm{A}_{\text{NS}_{\ell}}$
be two glueing maps. One obtains two glue-ups $\text{T}_{\ell}\oplus_{\varphi_{1}}\text{NS}_{\ell}$,
$\text{T}_{\ell}\oplus_{\varphi_{2}}\text{NS}_{\ell}$ and after identification
of these glue-ups with $U^{\oplus3}$ by using suitable isomorphisms,
we get two embeddings $\iota_{0},\iota_{1}:\text{T}_{\ell}\hookrightarrow U^{\oplus3}$
and $\tau_{0},\tau_{1}:\text{NS}_{\ell}\hookrightarrow U^{\oplus3}$.
\\
Let $g\in O(\text{T}_{\ell}),$ $\Phi\in O(U^{\oplus3})$ be such
that $\iota_{1}=\Phi\circ\iota_{0}\circ g$. Since $\Phi$ is an element
of the orthogonal group, $\Phi(\tau_{0}(\text{NS}_{\ell}))=\tau_{1}(\text{NS}_{\ell})$.
Therefore the glueing of $\text{NS}_{\ell}$ and $\text{T}_{\ell}$
is unique up to isomorphism (when $\ell<0$ we will see examples of
non-unique glue-ups). 

Let $A$ be an abelian surface with $T(A)\simeq\text{T}_{\ell}$.
The Hodge structure on $H^{2}(A,\ZZ)$ is isomorphic to a Hodge structure
$(V^{p,q}){}_{p+q=2}$ on $U^{\oplus3}$ such that $V^{1,1}\cap U^{\oplus3}=T_{\ell}^{\perp}$.
We thus obtain that there exists a unique $1$-dimensional family
of such Hodge structures. 
\begin{defn}
We call the corresponding one dimensional family of abelian surfaces
$A$ with these Hodge structures the moduli space of abelian surfaces
with quaternionic multiplication by $\OO$.
\end{defn}

It is the curve $(\mathcal{H}\cup\bar{\mathcal{H})}/\G(\OO^{*})$,
which has $2$ irreducible components if and only if $9|L_{X}^{2}$
(see Section \ref{subsec:Shimura-curves}) and one irreducible component
otherwise (here the surfaces $A$ are such that $L_{X}^{2}=\ell$,
for $X=\Km_{3}(A)$).

That implies that the set of $\End(A)$-ideals modulo principal ideals
has a unique class, since $(\cH\cup\bar{\cH})/\G(\OO^{*})$ is the
unique family of abelian surfaces with multiplication by $\End(A)\simeq\OO$. 

Take any $\ell=0$ or $2\mod6$. For a generic element $X=\Km_{3}(A)$
in $\cM_{\ell}$, let $\OO$ be the endomorphism ring of $A$. Up
to isomorphism, the ring $\OO$ is isomorphic to the order $\OO_{\mu}=\ZZ[j]\oplus\ZZ[j]\mu\t$
defined in Section 3, with $\frac{1}{2}L_{X}^{2}=\mu\bar{\mu}D_{H}$.
An element of the curve $\mathcal{X}(\OO)=(\cH\cup\bar{\cH})/\G(\OO^{*})$
is a pair $(A,\iota_{A})$, where $\iota_{A}:\OO_{\mu}\hookrightarrow\End(A)$
is a embedding. To such a pair, one can associate the pair $(A,G_{A})$,
where $G_{A}$ is the order $3$ symplectic group $G_{A}=\la\iota_{A}(j)\ra$.
Conversely, if $G_{A}=\la J_{A}\ra$ is an order $3$ symplectic automorphism
group on an abelian surface $A$, by Section \ref{subsec:The-endomorphism-ring}
and Corollary \ref{cor:order-determined-by-discri} there is an isomorphism
$\iota_{B}$ from an order $\OO_{\mu}$ (which is determined up to
isomorphism by $L_{X}^{2}$) to the endomorphism ring $\End(B)$ such
that $(A,G_{A})\simeq(B,\la\iota_{B}(j)\ra)$.

The degree of the forgetful map 
\[
\Psi:\mathcal{X}(\OO)=(\cH\cup\bar{\cH})/\G(\OO^{*})\to\cM_{\ell}
\]
which associates to $(A,\iota_{A})$ the generalized Kummer surface
$X=\Km_{3}(A,G_{A})$, where $G_{A}=\la\iota_{A}(j)\ra$, is therefore
the number of Kummer structures on a generic $X\in\cM_{\ell}$. If
$\OO$ is an Eichler order (then $\ell=2\mod6$ or $3||\ell$), the
map $\Psi$ is the quotient of $\mathcal{X}(\OO)$ (which is irreducible
in that case) by the Atkin--Lehner involutions. 

When $9|\ell$, the conjugation map $c_{\phi}:\g\to\phi\g\phi^{-1}$
by $\phi\in\OO_{\mu}=\OO$ still preserves $\OO$ (thus $\OO^{*}$).
Since the norm of $\phi$ is negative, the action of $\phi$ on $\cH\cup\bar{\cH}$
exchanges the two half planes $\cH,\bar{\cH}$. The homomorphism $\iota:\OO\to\End(A)$
is not equivalent to $\iota\circ c_{\phi}$, and the map $\Psi$ factors
through the quotient of $\mathcal{X}(\OO)$ by the involution on $\mathcal{X}(\OO)$
induced by $\phi$, which exchanges the two irreducible components
of $\mathcal{X}(\OO)$. The moduli space $\cM_{\ell}$ is therefore
irreducible.

\subsection{\label{subsec:Preliminaries-on-the}Preliminaries on the lattice
$U\oplus A_{2}$.}

In order to study the irreducible components of the moduli space
 of non-algebraic generalized Kummer surfaces, we need some preliminary
results on the lattice $U\oplus A_{2}$ introduced in Section \ref{subsec:The-N=0000E9ron-Severi-latofA}.
By results of Vinberg \cite{Vinberg}, the lattice $(U\oplus A_{2})(-1)$
is the Néron-Severi group of some K3 surfaces with a finite number
of automorphisms, which have been studied in \cite{ACR}. The following
elements $c_{1},\dots,c_{4}$ of $U\oplus A_{2}$
\[
c_{1}=(0,1,-1,1),\,c_{2}=(-1,-1,0,0),\,c_{3}=(1,0,0,-1),\,c_{4}=(1,0,1,0)
\]
(in basis $\g_{1},\dots,\g_{4}$ of $U\oplus A_{2}$) have square
$c_{j}^{2}=2$; their intersection matrix is {\footnotesize{}
\[
\left(\begin{array}{cccc}
2 & -1 & 0 & 0\\
-1 & 2 & -1 & -1\\
0 & -1 & 2 & -1\\
0 & -1 & -1 & 2
\end{array}\right).
\]
}Let us denote by $r_{1},\dots,r_{4}$ the order $2$ reflection 
\[
x\in U\oplus A_{2}\to x-(c_{k}x)c_{k}\in U\oplus A_{2},
\]
where $(\,,\,)$ denotes the intersection form between elements of
the lattice $U\oplus A_{2}$. The element $c_{5}=(0,0,1,1)\in U\oplus A_{2}$
has square $c_{5}^{2}=6$ and the reflection 
\[
r_{5}:x\to x-\tfrac{1}{3}(c_{5}x)c_{5}
\]
preserves the lattice $U\oplus A_{2}$. In \cite{Vinberg}, Vinberg
proves that lattice $U\oplus A_{2}$ (of signature $(3,1)$) has the
remarkable property that the group $W$ generated by the reflections
through the roots of square $2$ has finite index in the orthogonal
group, and that in fact $W$ is generated by $r_{1},\dots,r_{4}$.
Since the discriminant of the lattice $U\oplus A_{2}$ is $3$, only
reflections through elements of square $2$ and $6$ can preserve
$U\oplus A_{2}$. There is, up-to the action of $W$, a unique element
of square $6$ and the orthogonal group $O(U\oplus A_{2})$ is generated
by $r_{1},\dots,r_{5}$. We denote by $SO(U\oplus A_{2})$ the index
$2$ subgroup of $O(U\oplus A_{2})$ of elements with determinant
$1$; it is generated by the elements $r_{i}r_{j},\,1\leq i,j\leq5$. 

\subsection{The Barth moduli space of complex tori with order $3$ symplectic
groups}

The following construction is due to Barth and we take the same notations
as \cite[Section 2.1]{Barth}. Let $J_{A}$ be the generator of the
order $3$ symplectic group $G_{A}$ acting on $H_{1}(A,\ZZ)$. In
the basis $\a_{1},\b_{1},\a_{2},\b_{2}$ of $H_{1}(A,\ZZ)$, the action
of the symplectic order $3$ automorphism $J_{A}$ is by the matrix
\[
J=\left(\begin{array}{cccc}
0 & -1 & 0 & 0\\
1 & -1 & 0 & 0\\
0 & 0 & 0 & -1\\
0 & 0 & 1 & -1
\end{array}\right).
\]
Knowing the complex structure on $A$ amounts to fixing some $\RR$-linear
maps $z_{1},z_{2}:H_{1}(A,\RR)\to\CC$ inducing an isomorphism $H_{1}(A,\RR)\to\CC^{2}$
of real vector spaces. The coordinates $(z_{1},z_{2})$ determine
a period $\o=z_{1}\wedge z_{2}\in H^{2}(A,\CC)$. The period $\o$
is determined by the complex structure on $A$ up to multiplication
by a scalar, and $\o\wedge\o=0$, $\o\wedge\bar{\o}$ is positively
oriented, so that $(H^{2}(A,\CC),\CC\o)$ is a weight $2$ Hodge structure.
For the converse, let us recall Shioda's Torelli Theorem:
\begin{thm}
\label{thm:Shioda}\cite{Shioda} Let $A$ and $B$ be two $2$-tori.
There is an isomorphism of Hodge structures
\[
(H^{2}(B,\ZZ),\CC\o_{B})\simeq(H^{2}(A,\ZZ),\CC\o_{A})
\]
if and only if $B\simeq A\text{ or }B\simeq\hat{A}$, where $\hat{A}$
is the dual of $A$. For any given weight $2$ Hodge structure $(U^{\oplus3},\CC\o)$,
there exists a complex $2$-torus $A$ with an Hodge isometry 
\[
(H^{2}(A,\ZZ),\CC\o_{A})\simeq(U^{\oplus3},\CC\o).
\]
\end{thm}

Thus a space $\CC\o\subset H^{2}(A,\CC$) with $\o\wedge\o=0$ and
$\o\wedge\bar{\o}$ positively oriented determines (at most) two complex
structures (in \cite[Section 2.1]{Barth}, Barth did not address that
problem). Moreover if the map $J$ preserves $\o$, i.e. if $J^{*}\o=\o$,
then $J$ is $\CC$-linear for the complex structure defined by $\CC\o$
and $J_{A}$ is holomorphic. The group $G_{A}=\la J_{A}\ra$ then
acts on $\wedge^{2}H^{1}(A,\ZZ)=H^{2}(A,\ZZ)\simeq U^{\oplus3}$,
and the invariant part is the lattice $U\oplus A_{2}$. In \cite[Section 2.1]{Barth},
Barth considers the natural domain
\[
\Omega=\{\o\in\PP((U\oplus A_{2})\otimes\CC)\,|\,\o\wedge\o=0,\,\,\o\wedge\bar{\o}>0\},
\]
which he calls a period domain for pairs $(A,G_{A})$ of a complex
$2$-torus $A$ with an order $3$ automorphism group $G_{A}$. There
is a group $\G\subset O(U\oplus A_{2})$ (described below) acting
on $\Omega$, such that the quotient $\mathcal{M}_{B}=\Omega/\G$
(of dimension $2$; the action of $\G$ is not discrete) is a moduli
space for these pairs $(A,G_{A})$. This is done by associating for
each such $A$, with Hodge structure $(H^{2}(A,\ZZ),\omega_{A})$
and a marking $\iota:H^{2}(A,\ZZ)\stackrel{\simeq}{\to}U^{\oplus3}$,
the period $[\iota(\o_{A})]\in\Omega\subset\PP((U\oplus A_{2})\otimes\CC)$;
taking the period $\iota(\o_{A})$ modulo the group $\G$ makes the
map independent of the choice of the marking $\iota$. 

We remark that although a period $[\o]\in\PP((U\oplus A_{2})\otimes\CC)$
lifts to a unique period $\o$ in $\PP(U^{\oplus3}\otimes\CC)$ (we
recall that there is a natural inclusion $U\oplus A_{2}\subset U^{\oplus3}$),
according to Shioda's Torelli Theorem, there are (at most) two surfaces
$A,\hat{A}$ which have the same period $\o$. Of course when $A$
is a principally polarized abelian surface, one has $A\simeq\hat{A}$,
and there is no confusion, but in general the moduli space $\mathcal{M}_{B}$
parametrizes pairs $(A,G_{A}),(\hat{A},\hat{G}_{A})$ of two dimensional
complex tori with an order $3$ automorphism group.  

The groups $O(U\oplus A_{2}),SO(U\oplus A_{2})$ defined in Section
\ref{subsec:Preliminaries-on-the} are related to the group $\G$
as follows. As defined in \cite[Section 2.1]{Barth}, let $\G'$ be
the group of invertible matrices in basis $\a_{1},\b_{1},\a_{2},\b_{2}$
preserving the orientation of $H_{1}(A,\ZZ)$ and commuting with the
action of $J$. This is the group of elements {\footnotesize{}$\left(\begin{array}{cc}
A & B\\
C & D
\end{array}\right)$ }with matrices $A,B,C,D\in GL_{2}(\ZZ)$ that are commuting with
$T=\left(\begin{array}{cc}
0 & -1\\
1 & -1
\end{array}\right)$. The group $\G$ is $\G=\rho(\G')$, the image by the representation
$\rho:\G'\to O(U\oplus A_{2})$ of the group $\G'$ acting on $U\oplus A_{2}\subset U^{\oplus3}=\wedge^{2}H^{1}(A,\ZZ)$.
From that description, we obtain the following result:  
\begin{prop}
\label{prop:The-group-Gamma}We have $GL_{2}(\ZZ[j])\simeq\G'$ and
the kernel of $\rho$ is $\la-jI_{2}\ra$, where $j^{2}+j+1=0$. The
group $SO(U\oplus A_{2})$ is equal to the group $\G$ and is isomorphic
to $GL_{2}(\ZZ[j])/\la-jI_{2}\ra$. 
\end{prop}

So as expected by Barth, the group $\G$ is close to be the orthogonal
group $O(U\oplus A_{2})$.
\begin{proof}
(Of Proposition \ref{prop:The-group-Gamma}). The fact that $GL_{2}(\ZZ[j])\simeq\G'$
is obtained readily from the above description, by interpreting the
$2\times2$ matrices as block matrices, by substituting $T=-\left(\begin{array}{cc}
0 & -1\\
1 & -1
\end{array}\right)$ to $j$, and the size $2$ identity matrix to $1$. It is not difficult
to check that the kernel of $\rho$ contains the matrix{\small{} $-\left(\begin{array}{cc}
j & 0\\
0 & j
\end{array}\right)$}. The morphism $\rho:GL_{2}(\ZZ[j])\to O(U\oplus A_{2})$ extends
to a morphism $\rho_{\QQ}:GL_{2}(\QQ[j])\to GL((U\oplus A_{2})\otimes\QQ)$.
The kernel of $\rho_{\QQ}$ is a distinguished sub-group. The distinguished
sub-groups of $GL_{2}(\QQ[j])$ are classified: these are either the
group of homotheties, $SL_{2}(\QQ[j])$ or some subgroups containing
$SL_{2}(\QQ[j])$. Only the sub-group $\la-jI_{2}\ra$ of the group
of homotheties is in the kernel of $\rho$. The group $SL_{2}(\QQ[j])$
is not contained in the kernel of $\rho_{\QQ}$, thus neither are
his over-groups. We conclude that $\ker(\rho)=\la-jI_{2}\ra$ and
$\G\simeq GL_{2}(\ZZ[j])/\la-jI_{2}\ra$. The following matrices{\small{}
\[
\left(\begin{array}{cc}
-j & 0\\
0 & 1
\end{array}\right),\,\left(\begin{array}{cc}
1 & 1\\
0 & 1
\end{array}\right),\,\left(\begin{array}{cc}
1 & j\\
0 & 1
\end{array}\right),\,\left(\begin{array}{cc}
0 & 1\\
-1 & 0
\end{array}\right)
\]
}are generators of $GL_{2}(\ZZ[j])$ (see e.g. \cite{Swan}). A direct
computation gives that their action on $U\oplus A_{2}$ is by the
determinant $1$ matrices{\footnotesize{}
\[
\begin{array}{l}
\tau_{1}=\left(\begin{array}{cccc}
1 & 0 & 0 & 0\\
-1 & 1 & 2 & 1\\
-1 & 0 & 1 & 0\\
0 & 0 & 0 & 1
\end{array}\right),\,\tau_{2}=\left(\begin{array}{cccc}
1 & 0 & 0 & 0\\
-1 & 1 & -1 & -2\\
0 & 0 & 1 & 0\\
1 & 0 & 0 & 1
\end{array}\right),\end{array}
\]
\[
\tau_{3}=\left(\begin{array}{cccc}
1 & 0 & 0 & 0\\
0 & 1 & 0 & 0\\
0 & 0 & 1 & 1\\
0 & 0 & -1 & 0
\end{array}\right),\,\tau_{4}=\left(\begin{array}{cccc}
0 & -1 & 0 & 0\\
-1 & 0 & 0 & 0\\
0 & 0 & -1 & -1\\
0 & 0 & 0 & 1
\end{array}\right),
\]
}in basis $\g_{1},\dots,\g_{4}$ of $U\oplus A_{2}$, therefore $\tau_{1},\dots,\tau_{4}$
are generators of $\G$. One can check that 
\[
\begin{array}{l}
\tau_{1}=(r_{1}r_{3}r_{2})^{2}r_{4}r_{2},\,\,\tau_{2}=r_{1}r_{2}r_{4}r_{2}r_{1}r_{3}r_{2}r_{3},\\
\tau_{3}=r_{3}r_{1}r_{2}r_{5}r_{1}r_{4},\,\,\tau_{4}=r_{1}r_{3}r_{2}r_{3}r_{1}r_{2},
\end{array}
\]
which implies that $\G\subset SO(U\oplus A_{2})$. Conversely, the
products $r_{i}r_{j}$ with $1\leq i<j\leq5$ satisfy
\[
\begin{array}{c}
r_{1}r_{2}=\tau_{3}^{2}\tau_{2}\tau_{4},\,r_{1}r_{3}=\tau_{1}\tau_{3}^{2}\tau_{2}\tau_{4}\tau_{2},\,r_{1}r_{4}=\tau_{3}^{2}\tau_{2}\tau_{1}\tau_{4}\tau_{1},\,r_{1}r_{5}=\tau_{3}^{2}\tau_{2}\tau_{3}\\
r_{2}r_{3}=\tau_{4}\tau_{3}\tau_{4}\tau_{1}\tau_{4}\tau_{3},\,r_{2}r_{4}=\tau_{3}\tau_{4}\tau_{3}\tau_{1}\tau_{4}\tau_{1},\,r_{2}r_{5}=\tau_{3}\tau_{4}\tau_{3}^{2},\\
r_{3}r_{4}=\tau_{4}\tau_{3}^{2}\tau_{2}\tau_{3}\tau_{2}\tau_{4}\tau_{3},\,r_{3}r_{5}=\tau_{4}\tau_{3}(\tau_{4}\tau_{1})^{2},\,r_{4}r_{5}=\tau_{3}\tau_{3}\tau_{4}\tau_{2}\tau_{3}\tau_{4},
\end{array}
\]
and since $r_{i}r_{j}=(r_{j}r_{i})^{-1}$ (because $r_{k}^{2}=1$),
one also knows the products $r_{i}r_{j}$ with $i>j$. Therefore $SO(U\oplus A_{2})\subset\G$
and the two groups are equal. 
\end{proof}

\subsection{The irreducible components of the moduli space  of non-algebraic
generalized Kummer surfaces}

 In this section, we study the moduli spaces $\cM_{\ell}$ of generalized
Kummer surfaces $X=\Km_{3}(A)$ with non-zero invariant class $L_{X}\in\NS X)$
such that $L_{X}^{2}=\ell\leq0$. We give an algorithm to compute
a complete set of all embeddings classes of the lattice $\text{NS}_{\ell}$
into $U^{\oplus3}$ for any $\ell\leq0$ with $\ell=0$ or $2\mod6$.
Contrary to the case with $\ell>0$, such an embedding is, in general,
far from being unique up to isometry. Moreover, the number of elements
in the genus of $\text{NS}_{\ell}$ grows to infinity with $\ell$.
Indeed, for any integer $k$, by the Smith--Minkowski--Siegel mass
formula, there is only a finite number of definite quadratic forms
of rank $\geq3$ that have less than $k$ elements in their genus. 

Consider $P=(5,-4,3,-3)\in U\oplus A_{2}$: it has square $-22$ and
$c_{j}P=-1$ for $j\in\{1,\dots4\}$ (see Section \ref{subsec:Preliminaries-on-the}
for the definition of $c_{j}$). Let us define the negative cone $N$
as 
\[
N=\{x\in U\oplus A_{2}\,\,|\,\,x^{2}\leq0,\,\,Px\leq0\}
\]
($P$ realizes a choice between the two connected components of $\{x^{2}\leq0\}\setminus\{0\}$).
By \cite[Section 6]{Vinberg}, the polyhedral cone
\[
\Pi=\{x\in U\oplus A_{2}\,\,|\,\,xc_{j}\leq0,\,\,j=1,...,4\}\cap N
\]
is a fundamental domain for the action of the reflection group $W$
on the cone $N$. Thus the cone
\[
\Pi'=\{x\in U\oplus A_{2}\,\,|\,\,xc_{j}\leq0,\,\,j=1,...,5\}\cap N
\]
is a fundamental domain for the action of $O(U\oplus A_{2})$ on $N$.
Since the group $SO=SO(U\oplus A_{2})$ is an index $2$ sub-group
of $O(U\oplus A_{2})$, any reflection $r_{j}$ is such that the cosets
of $O(U\oplus A_{2})/SO$ are $SO,r_{j}SO$. We thus get that $\Pi'+r_{j}\Pi'$
is a fundamental domain for $SO,$ in particular:
\begin{prop}
The cone $\Pi=\Pi'+r_{5}\Pi'$ is also a fundamental domain for the
action of $SO(U\oplus A_{2})=\G$. 
\end{prop}

The Hilbert basis of $\Pi$ has $5$ elements
\[
\begin{array}{l}
w_{1}=(1,-1,1,-1),w_{2}=(2,-1,1,-1),w_{3}=(2,-2,1,-1),\\
w_{4}=(3,-3,2,-1),w_{5}=(3,-3,1,-2),
\end{array}
\]
which means that every element of the cone $\Pi$ is linear combination
with positive or zero integral coefficients of the vectors $w_{k},\,k=1,\dots,5$.
The extremal rays of $\Pi$ are $w_{1},w_{2},w_{4},w_{5}$. The intersection
matrix of the vectors $w_{k},\,k=1,\dots,5$ is {\footnotesize{}
\[
-\left(\begin{array}{lcccc}
0 & 1 & 2 & 3 & 3\\
1 & 2 & 4 & 6 & 6\\
2 & 4 & 6 & 9 & 9\\
3 & 6 & 9 & 12 & 15\\
3 & 6 & 9 & 15 & 12
\end{array}\right).
\]
}For $j\in\{1,2,4,5\}$, let $F_{j}$ be the cone generated by the
rays $w_{k},\,k\in\{1,2,4,5\}\setminus\{j\}$. The cones $F_{j}$
with $j\in\{1,2,4,5\}$ are the $4$ facets of $\Pi$; one can check
that the facets $F_{4},F_{5}$ are exchanged under the group $\G$.
By the above description, one can compute all representatives of the
polarizations $L_{A}$ (with $L_{A}^{2}\leq0$) of the moduli spaces
of complex tori in the polyhedral cone $\Pi$. We have to take into
account that one must avoid repetition from the facets $F_{4},F_{5}$,
which means (for example) to choose to exclude solutions that are
in $F_{5}\setminus F_{4}$. Moreover, by Proposition \ref{subsec:The-N=0000E9ron-Severi-latofA},
if we are looking for surfaces $X$ with $L_{X}^{2}=0\mod6$, we must
take only the solutions $L_{A}=(n_{1},n_{2},n_{3}+n_{4},n_{4})$ in
basis $\g_{1},\dots,\g_{4}$ with $gcd(n_{1},n_{2},n_{3},n_{4})=1$
and $gcd(n_{1},n_{2},n_{3},3)=1$. If we search for surfaces with
$L_{X}^{2}=2\mod6$, we must search solution of the form $(3n_{1},3n_{2},3n_{3}+n_{4},n_{4})$
with $gcd(3n_{1},3n_{2},3n_{3},n_{4})=1$. 

For example, the element $w_{1}$ of the Hilbert basis is the unique
element of the cone $\Pi$ of square $0$, therefore
\begin{cor}
The moduli space of non-algebraic generalized Kummer surfaces $X$
such that $L_{X}^{2}=0$ is irreducible.
\end{cor}

Let us give another example. The polarizations 
\[
L_{1}=(6,-5,4,-4),\,L_{2}=(15,-14,14,-14),\,L_{3}=(9,-7,7,-7)
\]
are in $\Pi$ and are the representatives of all polarizations $L_{A}$
of square $-28$ modulo the action of $\G$. Therefore the moduli
space $\cM_{-84}$ has exactly $3$ irreducible components. Corresponding
to $L_{1},L_{2},L_{3}$, there are three non-isomorphic embeddings
of the lattice $\NS A)$ in the even unimodular lattice $U^{\oplus3}$.
Knowing the classes $L_{j}$, one can compute these embeddings, and
the transcendental lattice $T_{j}$ of the complex tori $A$ with
invariant class (under $G_{A}$) $L_{j}$, $j\in\{1,2,3\}$. The lattices
$T_{2},T_{3}$ are isomorphic to $\NS A)(-1)$, and the lattice $T_{1}$,
which has a basis with Gram matrix {\footnotesize{}
\[
\left(\begin{array}{ccc}
4 & 2 & -2\\
2 & 6 & -3\\
-2 & -3 & 6
\end{array}\right),
\]
}is not isomorphic to $\NS A)(-1)$ (the minimum of the lattice $T_{1}$
is $4$, against $2$ for $T_{2},T_{3}$), but of course it is in
the same genus. 

Using the above algorithm, for an integer $k$, the following table
gives in line $\#L_{A}$ the number of irreducible components in the
moduli space $\cM_{-6k}$ of generalized Kummer surfaces $X$ such
that $3L_{A}^{2}=L_{X}^{2}=-6k$ (which is also the number of inequivalent
polarizations $L_{A}$ under $\G$), and in the line $\#G$ the number
of elements in the genus of $\NS A)$:

\begin{tabular}{|c|c|c|c|c|c|c|c|c|c|c|c|c|c|c|c|c|c|c|}
\hline 
$k$ & $1$ & $2$ & $3$ & $4$ & $5$ & $6$ & $7$ & $8$ & $9$ & $10$ & $11$ & $12$ & $13$ & $14$ & $15$ & $16$ & $17$ & $18$\tabularnewline
\hline 
$\#L_{A}$ & $1$ & $1$ & $2$ & $1$ & $2$ & $2$ & $2$ & $2$ & $3$ & $2$ & $3$ & $3$ & $3$ & $3$ & $4$ & $2$ & $4$ & $5$\tabularnewline
\hline 
$\#G$ & $1$ & $1$ & $1$ & $1$ & $2$ & $1$ & $2$ & $2$ & $2$ & $1$ & $3$ & $1$ & $3$ & $2$ & $2$ & $2$ & $4$ & $1$\tabularnewline
\hline 
\end{tabular}

The following table gives the same information for $k$ such that
$L_{X}^{2}=-6k+2$:

\begin{tabular}{|c|c|c|c|c|c|c|c|c|c|c|c|c|c|c|c|c|c|c|}
\hline 
$k$ & $1$ & $2$ & $3$ & $4$ & $5$ & $6$ & $7$ & $8$ & $9$ & $10$ & $11$ & $12$ & $13$ & $14$ & $15$ & $16$ & $17$ & $18$\tabularnewline
\hline 
$\#L_{A}$ & $1$ & $1$ & $1$ & $2$ & $2$ & $2$ & $3$ & $3$ & $3$ & $4$ & $2$ & $5$ & $4$ & $5$ & $3$ & $7$ & $3$ & $5$\tabularnewline
\hline 
$\#G$ & $1$ & $1$ & $1$ & $2$ & $1$ & $2$ & $1$ & $3$ & $2$ & $3$ & $2$ & $3$ & $2$ & $4$ & $2$ & $5$ & $2$ & $4$\tabularnewline
\hline 
\end{tabular}

Let us remark that the set of Hodge structures on $T(A)$ for a complex
torus $A$ with $\NS A)\simeq\text{NS}_{\ell}$ is $\PP^{1}$ for
$\ell<0$: 
\begin{lem}
Let be $L\in U\oplus A_{2}$ such that $L^{2}<0$. The domain $\Omega_{L}=\Omega\cap L^{\perp}$
is isomorphic to $\PP^{1}.$ 
\end{lem}

\begin{proof}
The lattice $L^{\perp}\subset U\oplus A_{2}$ is positive definite,
so there are real coordinates $x_{1},x_{2},x_{3}$ of $L^{\perp}\otimes\RR$
in which the intersection form is $x_{1}^{2}+x_{2}^{2}+x_{3}^{2}=0$.
Let $z_{1},z_{2},z_{3}$ be the corresponding coordinates on $L^{\perp}\otimes\CC$;
each $\o\in\Omega_{L}$ is such that $\o.\o=0$ i.e. $z_{1}^{2}+z_{2}^{2}+z_{3}^{2}=0$,
and the condition $\o.\bar{\o}=z_{1}\bar{z}_{1}+z_{2}\bar{z}_{2}+z_{3}\bar{z}_{3}>0$
is empty since $\o\in\PP((U\oplus A_{2})\otimes\CC)$. Therefore $\Omega_{L}$
is a smooth quadric in $\PP^{2}$ and $\Omega_{L}\simeq\PP^{1}$.
\end{proof}
The moduli space $\cM_{L}$ of complex tori $A$ such that $L$ is
in the Néron-Severi group of $A$ is the quotient of $\PP^{1}$ by
$\G_{L}$, the stabilizer of $L$ in $\G$. Since $\G_{L}$ must preserve
the positive definite lattice $\NS A)^{\perp},$ such a group is finite
(the automorphism group of $\NS A)^{\perp}$ is $(\ZZ/2\ZZ)^{2}\times S_{3}$
or the dihedral group $D_{6}$ of order $12$ according if $L_{X}^{2}=0\mod6$
or $L_{X}^{2}=2\mod6$). If $L$ is in the interior of the fundamental
domain $\Pi$, the group $\G_{L}$ is trivial. 

\subsection{\label{subsec:Concluding-remarks}Concluding remark}

\subsubsection*{Computing the number of polarizations up to automorphisms}

By a theorem of Narasimhan--Nori, there are always -up to automorphisms-
a finite number of polarizations of any given square on an abelian
variety $A$. The computation of that number for principal polarizations
has been done for several classes of abelian varieties by various
authors (Lange, Rotger...). Let now $A$ be an abelian surface with
an order $3$ symplectic automorphism group, and suppose (to simplify
things) that $A$ has a principal polarization. The generalized Kummer
structures on $X=\Km_{3}(A)$ are the conjugacy classes of such groups
$G_{A}$ on $A\simeq$$\hat{A}$. To $G_{A}$, we associate the polarization
$L_{A}\in\NS A)$, which generates the invariant part of $\NS A)$
under $G_{A}$ and the number $L_{A}^{2}$ does not depend on the
group $G_{A}$. If $L\in\NS A)$ is a polarization of square $L^{2}=L_{A}^{2}$,
its orthogonal complement in $\NS A)$ is a definite lattice of rank
$2$ and discriminant $3$, thus is the lattice $\mathbf{A}_{2}$.
Therefore by \cite[Section 4, Theorem]{Shioda}, $L$ corresponds
to the fixed part of an order $3$ automorphism. The integer $\frac{1}{2}e_{3}(\OO_{\mu})$
is therefore also the number of orbits under $\aut(A)$ of polarization
$L\in\NS A)$ of square $L^{2}=L_{A}^{2}$.

\noindent Xavier Roulleau,
\\Universit\'e d'Angers,
\\CNRS, LAREMA, SFR MATHSTIC,   
\\F-49000 Angers, France
\\ \email{xavier.roulleau@univ-angers.fr}
%
\begin{verbatim}



\end{verbatim}

\end{document}